\documentclass[11pt]{article}
\usepackage[top=1in, bottom=1in, left=1.25in, right=1.25in]{geometry}
\usepackage{amsmath,amssymb,amsthm,xcolor,enumerate,bbm}
\usepackage[unicode,breaklinks=true,colorlinks=true]{hyperref}

\usepackage{theoremref}
\usepackage{mathrsfs}
\usepackage[]{authblk}
\usepackage{comment}
\usepackage{cancel}
\usepackage{graphicx}

\usepackage{tikz}
\usepackage{pgfplots}

\numberwithin{equation}{section}
\newtheorem{thm}{Theorem}[section]

\newtheorem{lem}[thm]{Lemma}

\newtheorem{defn}[thm]{Definition}

\theoremstyle{remark}
\newtheorem{remark}[thm]{Remark}
\theoremstyle{definition}
\newtheorem{example}[thm]{Example}
\newcommand{\bke}[1]{\left ( #1 \right )}
\newcommand{\bkt}[1]{\left [ #1 \right ]}
\newcommand{\bket}[1]{\left \{ #1 \right \}}
\newcommand{\norm}[1]{\left \| #1 \right \|}
\newcommand{\cB}{\mathcal{B}}

\newcommand{\R}{\mathbb{R}}
\newcommand{\N}{\mathbb{N}}
\renewcommand{\div}{\mathop{\rm div}\nolimits}

\newcommand\La{\Lambda}

\newcommand{\si}{\sigma}

\newcommand\de{\delta}
\newcommand{\nb}{\nabla}
\newcommand{\lec}{{\ \lesssim \ }}
\newcommand{\gec}{{\ \gtrsim \ }}

\newcommand{\bka}[1]{{\langle #1 \rangle}}

\newcommand\al{\alpha}

\newcommand\e {\varepsilon}  %

\renewcommand\th{\theta}

\newcommand\la{\lambda}

\newcommand{\ZZ}{\mathbb{Z}}
\newcommand{\NN}{\mathbb{N}}

\newcommand{\esssup} {\mathop{\mathrm{ess\,sup}}}
\newcommand{\supp} {\mathop{\mathrm{supp}}}
\newcommand{\pv} {\mathop{\mathrm{p.v.\!}}}

\newcommand{\td}{\tilde}

\newcommand{\I}{\infty}
\newcommand{\oo}{\infty}
\newcommand{\Eq}[1]{\begin{equation*} #1 \end{equation*}}
\newcommand{\EQ}[1]{\begin{equation} #1 \end{equation}}
\newcommand{\EQS}[1]{\begin{equation}\begin{split} #1 \end{split}\end{equation}}

\newcommand{\EQN}[1]{\begin{equation*}\begin{split} #1 \end{split}\end{equation*}}
\newcommand{\EN}[1]{\begin{enumerate} #1 \end{enumerate}}
\newcommand{\uloc}{{\mathrm{uloc}}}
\newcommand{\loc}{{\mathrm{loc}}}
\newcommand{\far}{{\mathrm{far}}}

\usepackage{accents}

\newcommand{\BKT}{\mathring M^{2,2}}
\newcommand{\II}{\mathrm{II}}

\newcommand{\ph}{\varphi}

\begin{document}

\title{Global Navier-Stokes flows in intermediate spaces}

\author{Zachary Bradshaw, Misha Chernobai and Tai-Peng Tsai}
\date{}
\maketitle

\begin{abstract}

We construct global weak solutions of the three dimensional incompressible Navier-Stokes equations in intermediate spaces between the space of uniformly locally square integrable functions and Herz-type spaces which involve weighted integrals centered at the origin.
Our results bridge the existence theorems of  Lemari\'e-Rieusset and of Bradshaw, Kukavica and Tsai. An application to eventual regularity is included which generalizes the prior work of Bradshaw, Kukavica and Tsai as well as Bradshaw, Kukavica and Ozanski.

\medskip

\medskip

\end{abstract}

\setcounter{tocdepth}{1}
\renewcommand{\baselinestretch}{0.8}\normalsize
\tableofcontents
\renewcommand{\baselinestretch}{1.0}\normalsize

\section{Introduction}
Consider the incompressible Navier-Stokes equations in $\R^3$,
\EQS{\label{eq.NSE}
   &\partial_t u-\Delta u +u\cdot\nabla u+\nabla p = 0,
   \\& \nabla \cdot u=0,
   }
with divergence-free initial data $u_0$.
Global-in-time local energy solutions of \eqref{eq.NSE} for initial data which do not decay ``rapidly'' at spatial infinity was first constructed in local uniform $L^2$ spaces \cite{LR,KiSe,KwTs,MaMiPr}. It was then extended to data whose norms at a sequence of scales grow under certain rate (Herz spaces) \cite{Tsutsui,BT8,BKT,BKO,CW}, or is bounded in certain weighted spaces \cite{LR-Morrey,FDLR, Che}. They were also constructed
in the framework of Wiener amalgam spaces \cite{BT4,BLT} which are intermediate spaces between $L^2$, which is the initial data space for Leray's original solutions \cite{leray}, and $L^2_\uloc$.
The main feature of these solutions is that they satisfy the local energy inequality, the pressure satisfies the local pressure decomposition and they possess \textit{a priori} bounds in spaces tailored to the initial data.
The goal of this paper is to identify a continuous family of spaces between uniform local and Herz spaces where we can construct global-in-time local energy solutions.

For this purpose, we first reformulate the definitions of the local uniform $L^q$ space and the Herz spaces for comparison.
We may reformulate the definition of local uniform $L^q$ space as
\EQ{
L^q_\uloc=\bket{
f\in L^q_\loc(\R^3); \quad
\sup_{x\in \R^3} \int_{B_1(x)} |f|^q <\infty}.
}
Its subset of functions vanishing at spatial infinity is
\EQ{
E^q=\bket{
f\in L^q_\loc(\R^3); \quad
\lim_{|x|\to \infty} \int_{B_1(x)} |f|^q =0}.
}

Similarly, the Herz spaces considered in \cite{BKT} (see also discussion in \cite{FDLR2}) are, for $\la\ge 0$,
\EQ{\label{Mpq}
M^{q,\lambda} =\bket{ f\in L^q_\loc(\R^3), \quad {\int_{B_1(0)}|f|^q }+
\sup_{|x|>1}  \frac1{|x|^\la}\int_{B(x,  |x|/2)}|f|^q <\infty}.
}
Its subset of functions vanishing at spatial infinity is
\EQ{\label{Mpqring}
\mathring M^{q,\lambda} =\bket{  f\in L^q_\loc(\R^3), \quad
\lim_{|x|\to \infty} \frac1{|x|^\la}\int_{B(x,  |x|/2)}|f|^q =0}.
}

Global existence was shown in $\mathring M^{2,\lambda}$ for $0\le \la \le 2$ in \cite{BKT}. This built upon a local existence result in \cite{BK}. Although details are only given for $\la=2$, the other cases $\la<2$ can be proved in the same way. Observe that $L^r(\R^3)\subset \BKT$ if and only if $r \le 6$. This problem was studied in 2D by Basson \cite{Basson}.
We now consider intermediate spaces between $E^2$ and $M^{2,\lambda}$. They
must be subsets of $L^2_\loc(\R^3)$ so that the local energy inequality may be valid up to time zero, and we may get a priori estimates in terms of the initial data.

\begin{defn}\label{Fpalla.def}
Let $1\le p<\infty$,
 $0\le \al \le 1$ and $0\le \la<\infty$. Define $F^p_{\al,\la}$ to be the space of $f\in L^p_\loc(\R^3)$ with finite norm $\norm{f}_{F^p_{\al,\la}(R)}$ where for any $R\ge 1$,
\EQ{\label{Fal.def}
\norm{f}_{F^p_{\al,\la}(R)} = \sup_{x\in\R^3} \bke{\frac1{r_x^\la}\int_{B(x, r_x)} |f|^p }^{1/p}, \quad r_x =\frac 1 8\max\{R,|x|\}^\al.
}
We denote by
 $F^p_{\al,\la,\si}$ the space of divergence-free vector fields in $\R^3$ with all components in $F^p_{\al,\la}$.
Furthermore, we denote by
$\mathring F^p_{\al,\la}$ and $\mathring F^p_{\al,\la,\si}$
the closure of $C^\infty_c(\R^3)$ and $C^\infty_{c,\si}(\R^3)=\{ \zeta\in C^\infty_c(\R^3;\R^3), \div \zeta=0\}$ in $F^p_{\al,\la}$-norm, respectively.
\end{defn}

\emph{Comments on Definition \ref{Fpalla.def}:}
\EN{
\item We will show in Lemma \ref{vanishing.initial.data} that finiteness of $\norm{f}_{F^p_{\al,\la}(R)}$ implies finiteness of $\norm{f}_{F^p_{\al,\la}(\rho)}$ for any $R,\rho\ge1$. We will show in Lemma \ref{F_hatspaces} that the space
$\mathring F^p_{\al,\la,\si}$ is the same as the space of diverence-free vector fields with all components in $\mathring F^p_{\al,\la}$.

\item Note that when $\al=0$, $r_x=1/8$, for any $\la\ge0$,
\[
F_{0,\la}^2  = L^2_\uloc, \quad \mathring F_{0,\la}^2  =E^2,
\]
and when $\al=1$, $r_x=\frac 1 8\max(R,|x|)$,
\[
F_{1,\la}^2  = M^{2,\la}, \quad \mathring F_{1,\la}^2  = \mathring M^{2,\la}.
\]
If $\la=0<\al$, the space $F_{\al,0}^2$ is strictly contained in $L^2_\uloc$ as its norm involves bigger balls.
We do not consider $\al>1$, in which case the ball $B(x,r_x)$ will cover the origin for $|x|$ large enough. Hence it is not localized around $x$.

\item For fixed $\al$, $\la$ and $R$, the family of balls $\{B(x, r_x): \ x \in \R^3\}$ has the important property that, if $B(x, r_x) \cap B(y, r_y)$ is nonempty, then $r_x$ and $r_y$ are comparable, no matter whether $|x| \lec R$ or $|x| \gg R$.

\item Simple examples show that $F_{\al,\la}^p$ is not a subset of either $L^2_\uloc$ or $M^{2,\lambda}$ for $0<\al<1$,
see Examples \ref{exampleS2a} and \ref{exampleS2b}. Hence $F_{\al,\la}^p$ is
a \emph{new} solution class. %
}

We now introduce the notion of local energy solutions, similarly to \cite{BK,BKT}.
Unlike \cite{LR,KiSe,KwTs}, the following definition does not assume uniform in space bounds of the local energy.
For a given ball $B$ of radius $r$, let $B^*$ and $B^{**}$ denote concentric balls of radii $\frac 98r$ and $\frac 54r$, respectively.

\begin{defn}
  A vector field $u\in L^2_{\loc}(\R^3\times [0,T))$, $T>0$, is a \textbf{local energy solution} to \eqref{eq.NSE} with initial data $u_0\in L^2_{\loc}(\R^3)$ if the following holds:
  \begin{itemize}
    \item $u\in\cap_{R>0}L^{\infty}(0,T;L^2(B_R(0))),~\nabla u\in L^2_{\loc}(\R^3\times[0,T))$,
    \item there is $p\in L^{\frac32}_{\loc}(\R^3\times[0,T))$ such that ${u,p}$ is NSE solution in the sense of distributions,
    \item For all compacts $K\subset\R^3$ we have $u(t)\rightarrow_{t\rightarrow0} u_0$ in $L^2(K)$,
    \item $u$ is a Caffarelli-Kohn-Nirenberg solution, for all $\phi \in C^{\infty}_0(\R^3\times (0,T)),~\phi\ge0,$
    \EQ{
    2\int\int|\nabla u|^2\phi ~dx~dt\le\int\int |u|^2(\partial_t\phi+\Delta \phi)~dx~dt+\int\int(|u|^2+2p)(u\cdot \nabla\phi)~dx~dt,
    }
    \item the function $t\mapsto\int u(x,t)w(x)~dx$ is continuous on $[0,T)$ for any $w\in L^2(\R^3)$ with a compact support,
    \item for every ball $B\subset \R^3$, there exists $p_B(t)\in L^{\frac 32}(0,T)$ such that  for $x\in B^*$ and $0<t,T$,
\EQS{
\label{p.dec}
p(x,t)-p_B(t) &=   - \frac 13 |u|^2(x,t) +\pv \int_{ y\in B^{**}}  K_{ij}(x-y)   u_i u_j(y,t)\,dy
\\&\quad
+  \int_{ y\notin B^{**}}  \bkt{ K_{ij}(x-y) - K_{ij}(x_B - y  )}  u_i u_j(y,t)\,dy,
}
    where $x_B$ is the center of $B$ and $K_{ij}=\partial_i\partial_j(4\pi|y|)^{-1}$.
  \end{itemize}

It is a local energy solution for $0\le t<\infty$ if it is a local energy solution for $0\le t<T$ for any $T<\infty$.
\end{defn}

The following is our main theorem on the local and global existence of local energy weak solutions in $F^2_{\al,\la}$.

\begin{thm} \label{mainthm}
Assume
\EQ{\label{alla-cond}
0\le \al\le 1, \quad 0\le \la\le 2.
}
For any divergence-free initial data $u_0\in \mathring F^2_{\al,\la}$, there exist $T_+>0$ (with an explicit lower bound) and a local energy solution $(u,p)$ to the Navier-Stokes equation \eqref{eq.NSE} with initial data $u_0$ for $t \in (0,T_+)$ such that $u(t) \in C(0,T; F^2_{\al,\la}(R))$ for any $0<T<T_+$ and $1<R<\infty$.

If we further assume
\EQ{\label{vanishing}
\lim_{\rho\to \infty} \rho^{\al(\la-2)} \|u_0\|^2_{F^{2}_{\al,\lambda}(\rho)} =0
}
when $(3-\la)\al<1$, $($no condition when $(3-\la)\al\ge 1)$, then we can take $T_+=\infty$, the solution is globally defined and satisfies the a priori bound \eqref{apriori}. %
\end{thm}

\emph{Comments on Theorem \ref{mainthm}:}

\EN{
\item It will follow from Lemma \ref{vanishing.initial.data} that for $\|u_0\|^2_{F^{2}_{\al,\lambda}(R=1)}=A$,
\[
\rho^{\al(\la-2)} \|u_0\|^2_{F^{2}_{\al,\lambda}(\rho)} \lec \rho^{\al(\la-2)}  \rho ^{\al(1-\al)(3-\la)}A
= (\rho^\al)^{1-(3-\la)\al}A.
\]
Hence \eqref{vanishing} is automatic if $(3-\la)\al>1$ (right region of Figure \ref{alla-cond-png}). It can be proved when $(3-\la)\al=1$ using $u_0\in \mathring F^2_{\al,\la}$, see Lemma \ref{lem2.9}. 
When $(3-\la)\al<1$ (left region of Figure \ref{alla-cond-png}), Example \ref{exampleS2b} gives functions $u_0\in \mathring F^2_{\al,\la}$ that fail \eqref{vanishing}, and hence \eqref{vanishing}
has to be assumed for global existence.

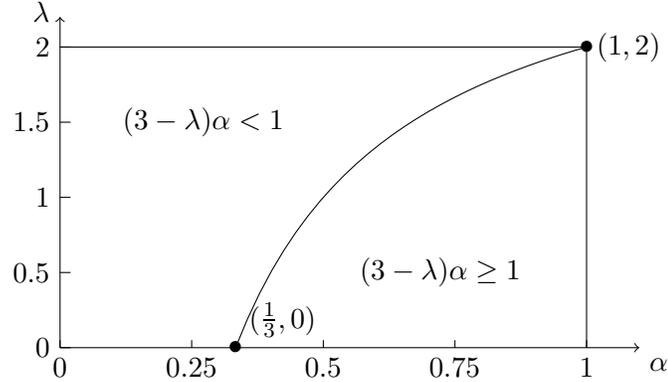
\begin{figure}[h]
\begin{center}
\begin{tikzpicture}[xscale=7, yscale=2]
  \draw[->] (0,0) -- (0,2.2);
  \draw[->] (0,0) -- (1.1,0);
    \draw[-] (1,0) -- (1,2);
  \draw[-] (0,2) -- (1,2);
  \node[above left] at (0,2.1) {$\lambda$};
  \node[below right] at (1.1,0) {$\alpha$};
  \foreach \x in {0,0.25,0.5,0.75,1} {
    \draw (\x,0) -- (\x,0.05);
    \node[below] at (\x,0) {\x};
  }
\foreach \x in {0,0.5,1,1.5,2} {
     \draw (0,\x) -- (0.02,\x);
     \node[left] at (0,\x) {\x};
  }

\node [black] at (1,2) {\textbullet};
\node[right] at (1,2) {$(1,2)$};
\node [black] at (1/3,0) {\textbullet};
\node[above right] at (.34,0) {$(\frac13,0)$};

  \clip (-0.25,0) rectangle (1.1,2.1);

\draw plot [
        samples=30,
        domain=.333:1
        ] (\x,{3-1/(\x)});

 \node[right] at (0.55,0.5) {$(3-\la)\al\ge1$};
 \node[right] at (0.1,1.5) {$(3-\la)\al < 1$};
\end{tikzpicture}

\end{center}
\caption{The $(\al,\la)$ parameter region for global existence: For $(\al,\la)$ in the left part we need to assume \eqref{vanishing}. For $(\al,\la)$ in the right part we don't.}
\label{alla-cond-png}
\end{figure}

\item
Suppose $u_0$ satisfies \eqref{vanishing}.
This implies
\[
\frac {\rho^{\al(\la - 2)}} {\rho^{\al \la }}\int_{B(0,\rho^\al)} |u_0|^2 \,dx  = \frac 1 {\rho^{2\al}}\int_{B(0,\rho^\al)} |u_0|^2 \,dx \to 0.
\]
Using \cite[Lemma 2.1]{BKT}, it implies $u_0\in \mathring M^{2,2}(\R^3)$. Time-global existence is already known in this space in  \cite{BKT}. The main merit of Theorem \ref{mainthm} is to construct global weak solutions in the more restrictive class $F^2_{\al,\la}$.

\item Our construction is based on a priori bounds, whose proof essentially uses the condition $\la \le 2$. See Lemma \ref{a.priori} and Remark \ref{la<2}.

\item Our proof works uniformly for $0<\al\le 1$ and $0 \le \la \le 2$. When $\al=0$ and $F^2_{0,\la}= L^2_\uloc$, our proof only gives a local in time solution and not the global existence, as the time upper limit $T_n$ defined \eqref{Tn.def} cannot go to infinity as $n \to \infty$ when $\al=0$.

\item If $\la=0<\al$, the space $F_{\al,0}^2$ is strictly contained in $L^2_\uloc$ as its norm involves bigger balls. Hence solutions in $F_{\al,0}^2$ have more control.

}

The proof is based on
\emph{a priori} estimates of the solution in $F^p_{\al,\la}(R)$ for a sequence $R_n=2^n$, $n \in \NN_0$. This type of estimate extends ideas in \cite{JiaSverak-minimal}.
We denote the set of balls
\EQ{\label{cCn.def}
\cB_n = \bket{ B(z,r_z): \ r_z=\frac18\max(2^n,|x|)^\al}.
}
Define
\EQ{\label{aln.def}
a_{n}(t) =   \| u(t)\|_{ F^2 _{\al,\la}(R=2^n)}^2
= \sup_{B\in \cB_n} \frac 1 {|B|^{\la/3}}\int_B |u(x,t)|^2\,dx
}
and
\EQ{\label{ben.def}
b_{n}(t)=\sup_{B\in \cB_n} \frac 1 {|B|^{\la/3}}\int_0^t\int_{B}|\nb u(x,s)|^2\,dx\,ds.
}
Our goal is to show a family of  a priori estimates
\EQ{
\esssup_{0<t<T_n} a_{n}(t) + b_{n}(T_n)  \le C a_{n}(0), \quad n \in \NN,
}
with $C$ independent of $n$, for some $T_n \to \oo$.

The strategy of the proof is following: we establish pressure estimates in spaces $F^p_{\al,\la}$ first. To do that we use the pressure decomposition formula \eqref{p.dec} for any ball $B$.

Next we separate pressure in several parts $p= p_\text{near}(x,t)+ p_\text{far}(x,t)$ and apply Calderon-Zygmund inequality combined with estimates of the solution over annulus $B_{2^n}\setminus B_{2^{n-1}}$. We use local energy inequality combined with Gronwall lemma and pressure estimate to get a priori bound for solutions in $F^p_{\al,\la}$. The first step would be to obtain the following inequality for norm $a_n(t)$,
 \EQN{
a_n (t) \lesssim a_n(0) +C_1(n)	 \int_0^t a_n(s)\,ds +C_2(n)	\int_0^t a_n^3 \,ds.
}
After that we apply Gronwall lemma to get uniform bound up to some time $T$ which will depend on norm $\|u_0\|_{F^p_{\al,\la}(R)}$ and $R$. This allows us to prove local in time existence for all parameters $\al,\la$ and under conditions \eqref{alla-cond} we get the global existence. The construction of local energy solution is done by approximation, where we first approximate the initial data with $L^2$ sequence $\{u^n_0\}$, for every $n$ we have global energy solutions $u^n$ corresponding to initial data $u_0^n$ and after applying a priori bounds in $F^p_{\al,\la}$ we can construct a subsequence of $\{(u^n,p^n)\}$ that will converge to local energy solution $(u,p)$.

\subsection{Eventual regularity}

The weighted local spaces considered in \cite{BT8,BKT,BKO} have a natural application to the problem of eventual regularity. It is classical that solutions in the Leray class eventually regularize in the sense that there exists a time $T$ so that the solution in view is smooth on $\R^3\times (T,\I)$. This follows  from the fact that, for all $T>0$,
\EQ{\label{ineq.Leray}
\int_0^T\int  |\nb u|^2\,dx\,dt \leq  \|u_0\|_2^2.
}
Since this holds as $T\to \I$, by Sobolev embedding the $L^6$ norm must become small at some large times. Since $u$ is also bounded in $L^2$, interpolation implies $u$ is small in $L^3$ at certain large times. The small-data global well-posedness theory in $L^3$ and weak-strong uniqueness together guarantee the solution becomes smooth for all large times.
On the other hand, eventual regularity  is \textit{not} known for local energy solutions. This is due to the lack of an estimate that extends to arbitrarily large times. Indeed, Lemari\'e-Rieusset's solutions, which have data in $L^2_\uloc$ have only a short time estimate \cite{LR,LR2}.  The smaller $M^{2,q}$ and $\mathring F^{2}_{\al,\la}$ spaces, on the other hand, allow for an estimate in Theorem \ref{uniform.bound} which extends up to arbitrarily large times for certain choices of parameters.
These \textit{a priori} bounds can be used as a substitute for \eqref{ineq.Leray} in proving eventual regularity, as has been explored in \cite{BT8,BKT,BKO}.  In \cite{BKT} eventual regularity is proven above a paraboloid for data in $\mathring M^{2,1}(\R^3)$. In \cite{BKO} corresponding estimates are carried out in $\mathring M^{2,q}(\R^3_+)$ where $q\leq 1$. In this case, the paraboloid region can be replaced by a larger algebraic region with exponents depending on $q$. As was known to the authors of \cite{BKO}, the theorem in \cite{BKO} can be improved in  $\mathring M^{2,q}(\R^3)$ yielding an algebraic region of the form
\[
\{(x,t):  |x|^{2q} \lesssim t \text{ and } t>M  \},
\]
for a non-universal constant $M=M(u_0)$.\footnote{A proof is not written down but it follows directly from the argument in \cite{BKO}.}  Note that  the classical eventual regularity result for Leray weak solutions would correspond to $q=0$.

\begin{figure}
\begin{center}
 \begin{tikzpicture}[scale=4]

  \draw[->] (0,0) -- (0,1.1);
  \draw[->] (0,0) -- (1.1,0);

  \node[above left] at (0,1.1) {$\lambda$};
  \node[below right] at (1.1,0) {$\alpha$};

  \foreach \x in {0,0.25,0.5,0.75,1} {
    \draw (\x,0) -- (\x,0.05);
    \draw (0,\x) -- (0.05,\x);
    \node[below] at (\x,0) {\x};
    \node[left] at (0,\x) {\x};
  }

  \clip (0,0) rectangle (1.1,1.1);

\draw plot [
        samples=30,
        domain=.333:1
        ] (\x,{.33/(\x)});

 \node[right] at (0.333,1) {$\mathring {\mathfrak F}^{2}_{\al\la,1}$};  %
\node [black] at (.333,1) {\textbullet};

\node at (.8,.8) {$\boxed{\al \la = 1/3}$};

 \node[below] at (1,.333) {$\mathring {\mathfrak F}^{2}_{1,\al\la}$};%
\node [black] at (1,.333) {\textbullet};

 \node at (.45,.45) {$\mathring {\mathfrak F}^{2}_{\al\neq 1,\la\neq1}$};

\end{tikzpicture}

\caption{The spaces $\mathring {\mathfrak F}^{2}_{\al,\la} $ when $\al\la =1/3$, which is fixed for illustrative purposes.  Note that $\mathring M^{2,\al\la}=\mathring {\mathfrak F}^{2}_{1,\al\la}$ by Lemma \ref{Mpq-centered}. These spaces imply eventual regularity in regions with the same algebraic bounds, up to  constants. Spaces are strictly increasing in $\la$. The conclusion of Theorem \ref{thm.ER} is new in all but the points  $(\al=1,\la=1/3)$.}
\label{figure.ER}
\end{center}
\end{figure}
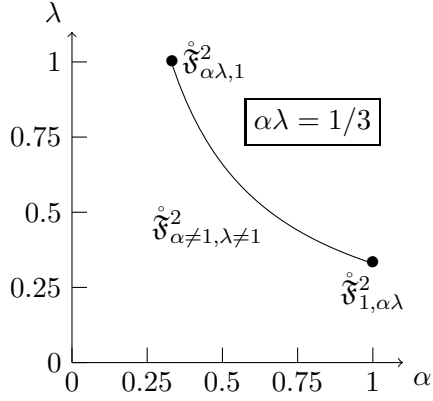

It is natural to ask what happens if we seek eventual regularity within the finer $F^2_{\al,\la}$ scale of spaces, noting that $M^{2,q}$ corresponds to $F^2_{1,q}$. For this we need to introduce new sub-classes of  our $F^2_{\al,\la}$ spaces. Let
\EQ{\label{Ffrak.def}
\mathfrak F^{2}_{\al,\la} := \{f \in F^{2}_{\al,\la}: \sup_{R\geq 1} \|f\|_{  F^{2}_{\al,\la}(R)  }<\I  \},
}
and
\EQ{\label{Ffrak0.def}
\mathring {\mathfrak F}^{2}_{\al,\la} := \{f \in \mathfrak F^{2}_{\al,\la}: \lim_{R\to \I}  \|f\|_{  F^{2}_{\al,\la}(R)  }=0  \}.
}
The latter is necessarily a subspace of $\mathring M^{2,\la}$, which can be easily seen using an alternative definition of the $M^{2,q}$ norm given in \cite{BKT}.

\begin{thm}[Eventual regularity]\label{thm.ER}
Assume that $ u_0 \in \mathring {\mathfrak F}^{2}_{\al,\la}$ where   $\al\in (0,1]$, $\la\in (0,1]$.   Let $(u,p)$ be a local energy solution with initial data $u_0$ on $\R^3\times(0,\I)$, satisfying%
\EQ{\label{thm.ERcond}
\esssup_{0<s<t}a_n(s)+b_n(t)< \infty, \quad \text{for all } t<\I,
}   for one fixed $n\in\NN_0$.
Then, there exists $M = M(u_0)$ so that $u$ is regular within a region of the form
\[
\{ (x,t)  :  |x|^{2 \al \la} \leq t, \, t\geq M \},
\]
where the suppressed constants depend on $\la$ and $\al$,  while $M$ is determined by the length scale $\rho$ at which $\|u_0\|_{F^2_{\al,\la}(\rho)}$ becomes small compared to a universal small constant. {Additionally, for $(x,t)$ in the preceding region, $|u(x,t)|\lesssim t^{-1/2}$.}
\end{thm}

The condition \eqref{thm.ERcond} is \eqref{eq4.6} and is needed to ensure that the \textit{a priori} bound \eqref{apriori} holds. This is satisfied by the solutions constructed in Theorem \ref{mainthm}.

The region of regularity here is \textit{already established} for $\mathring M^{2,\al \la}$-solutions. Our result identifies larger classes which have the same regularity region, insofar as the degree of the algebraic boundary is concerned, as the $\mathring M^{2,\al \la}$ case. To support this, in Lemmas \ref{Mpq-centered} and \ref{Remark.ER} we show that if $\al \la = q\in (0,1]$, then
\[
\mathring {\mathfrak F}^{2}_{1,q} = \mathring M^{2,q }\subsetneq \mathring {\mathfrak F}^{2}_{\al\neq 1,\la\neq 1} \subsetneq \mathring {\mathfrak F}^{2}_{q,1}.
\]
The largest space in Theorem \ref{thm.ER} which admits a regularity region with $ |x|^{2 \al \la} \lesssim t$ for large $t$ is, therefore, $\mathring {\mathfrak F}^{2}_{q,1}$.
This is illustrated in Figure \ref{figure.ER}.

\bigskip
The rest of the paper is structured as follows:
In Section \ref{sec2} we study properties of the intermediate spaces $F^{p}_{\al,\lambda}$.
In Section \ref{sec3} we show the pressure estimate.
In Section \ref{sec4} we show the a priori bounds.
In Section \ref{sec5} we prove the main theorem on the local and global existence of weak solutions.
Finally in Section \ref{sec6} we study eventual regularity.

\section{The intermediate spaces}\label{sec2}
In this section we explore
properties of spaces $F^{p}_{\al,\lambda}$ defined in Definition \ref{Fpalla.def}.
For $\rho>0$, we denote the ball and annulus
\[
B_\rho=B(0,\rho),\quad
A_\rho = B_{2\rho} \setminus B_\rho.
\]

\begin{lem}[annulas covering]\label{annulas-covering}
Fix $\al \in [0,1]$, $R \ge 1$, and let $r_x = \frac18 \max(R,|x|)^{\al}$ for $x \in \R^3$.  For any $\rho\ge R$, the annulus $A_\rho$ can be covered by $N(\rho)$ balls $B(x_k,r_{x_k})$, $x_k \in A_\rho$, $k=1,\ldots,N(\rho)$,
\[
A_\rho \subset \cup_{k=1}^{N(\rho)}  B(x_k,r_{x_k}),
\quad
\text{with}
\quad
N(\rho) \le C \rho^{3-3\al},
\]
where $C$ is independent of $R$ and $\rho$. Similarly, the ball $B_R$ can be covered by $N(R)$ balls $B(x_k,r_{x_k})$, $x_k \in B_R$, $k=1,\ldots,N(R)$, with $N(R) \le C R^{3-3\al}$.
\end{lem}

\begin{proof} Denote the lattice
\[
\La=r\ZZ^3, \quad r := \tfrac1{16} \rho^{\al}\le \tfrac12 \inf_{x \in A_\rho} r_x.
\]
Any point $y \in A_\rho$ belongs to at least one cube of side length $r$ and vertices in $\La$, and at least one of these vertices, denoted as $x$, is in $A_\rho$. Thus $|y-x| \le  \sqrt 3 r$ and hence
\[
y\in \overline B(x, \sqrt 3 r) \subset B(x,r_x).
\]
We choose $x_k$, $k=1,\ldots,N(\rho)$, to be all points in $\La \cap A_\rho$.
Let $m = \lceil 2\rho/r \rceil \le C\rho^{1-\al}$. We have $\La \cap A_\rho \subset \La_\rho$, where $\La_\rho = \{ r(z_1,z_2,z_3): z_i \in \ZZ, |z_i| \le m, i=1,2,3\}$. Hence
\[
N(\rho) \le \# \La_\rho=(2m+1)^3 \le C \rho^{3-3\al}.
\]
The proof for the ball $B_R$ is similar.
\end{proof}

\begin{lem}\label{vanishing.initial.data}
 Let $1\le p<\infty$, $0 \le \al \le 1$ and $0 \le \la < 3$.
 Let $1\le R < \rho<\infty$. Then $\norm{f}_{F^{p}_{\al,\lambda}(R)}<\infty$ if and only if $\norm{f}_{F^{p}_{\al,\lambda}(\rho)}<\infty$, and
\EQ{\label{lem2.1}
C_1 \norm{f}^p_{F^{p}_{\al,\lambda}(R)} \le \norm{f}^p_{F^{p}_{\al,\lambda}(\rho)}
 \le C_2 \norm{f}^p_{F^{p}_{\al,\lambda}(R)}
}
for $C_1=\bke{\frac \rho R}^{-{\al \la}}$ and $C_2 = C \bke{\frac \rho{R+\rho^\al}}^{\al(3-\la)}\lec \rho ^{\al(1-\al)(3-\la)}$.
\end{lem}

\emph{Remark.}\ This lemma shows that the space $F^{p}_{\al,\lambda}$ defined by the $F^{p}_{\al,\lambda}(R)$-norm does not depend on the choice of $R$.
Similar bound \eqref{lem2.1} hold for $\la\ge3$, but the formula for $C_2$ will change.%

\begin{proof} Denote $R_x=\frac18\max(R,|x|)^\al$ and $\rho_x=\frac18\max(\rho,|x|)^\al$. Note $R_x\le \rho_x \le (\rho/R)^\al R_x$.

Suppose $\norm{f}_{F^{p}_{\al,\lambda}(\rho)}<\infty$. For any $x\in\R^3$,
\[
\frac{1}{R_x^{\la}}\int_{B(x,R_x)}|f|^p\,dy\le \frac{\rho_x^{\la}}{R_x^{\la}}\frac{1}{\rho_x^{\la}}\int_{B(x,\rho_x)}|f|^p\,dy \le \bke{\frac \rho R}^{\al \la} \norm{f}_{F^{p}_{\al,\lambda}(\rho)}^p.
\]
This shows the first inequality of \eqref{lem2.1} with $C_1 = \bke{\frac \rho R}^{-{\al \la}}$.

Suppose now $\norm{f}_{F^{p}_{\al,\lambda}(R)}<\infty$. We may assume $\norm{f}_{F^{p}_{\al,\lambda}(R)}=1$.

If $|x|>\rho$, $\rho_x=R_x$, and
$
\frac{1}{\rho_x^{\la}}\int_{B(x,\rho_x)}|f|^p\,dy \le \norm{f}_{F^{p}_{\al,\lambda}(R)}^p=1$.

It remains to consider $|x|<\rho$ with $\rho_x= \frac 18 \rho^\al$.
If $2\rho_x < |x| < \rho$, all $y \in B(x,\rho_x)$ has $|y| \sim |x|$ and hence $R_y \sim R_x$. Hence $B(x,\rho_x)$ can be covered by $K$ balls $B(y_k,R_{y_k})$ with $K \lec (\rho_x/R_x)^3$, and
\EQS{\label{0920}
\frac{1}{\rho_x^{\la}}\int_{B(x,\rho_x)}|f|^p\,dy
&\le \frac{1}{\rho_x^{\la}}\sum_{k=1}^K \int_{B(y_k, R_{y_k})}|f|^p\,dy
\le \frac{1}{\rho_x^{\la}}\sum_{k=1}^K R_{y_k}^\la
\lec (\rho_x/R_x)^{3-\la}
\\
&\approx \bke{\frac \rho{R+|x|}}^{\al(3-\la)}
\lec \bke{\frac \rho{R+\rho^\al}}^{\al(3-\la)}.
}

If $|x|<2\rho_x =  \frac 14 \rho^\al$, then $B(x,\rho_x)\subset B(0,\rho^\al)$. In the case $R>\rho^\al$,
$R_y = \frac 18 R^\al$ in $B(0,\rho^\al)$, and $B(x,\rho_x)$ can be covered by $K$ balls $B(y_k,R_{y_k})$ with $K \lec (\rho_x/R_x)^3$. We have the same bound \eqref{0920}.

In the case $R<\rho^\al$, we cover $B(0,\rho^\al)$ by $B_R(0)$ and $K$ annulus
\[
\mathcal A_k = \{ y \in \R^3: 2^{k-1} R \le |y|< 2^{k}R\}, \quad k=1,\ldots,K,
\]
with $2^{K-1} R \le \rho^\al< 2^{K}R$. Each $\mathcal A_k$ can be covered by $N_k$ balls $B_{k,j}=B(y_{k,j}, R_{y_{k,j}})$ with $y_{k,j} \in \mathcal A_k$, $R_{y_{k,j}} \sim ( 2^k R)^\al $, and $N_k \lec ( 2^k R)^{3(1-\al)} $. Similarly, $B_R(0)$ can be covered by $N_0$ balls $B_{0,j}$ with $j \le N_0 \sim N_1$.
Thus
\EQN{
\frac{1}{\rho_x^{\la}}\int_{B(x,\rho_x)}|f|^p\,dy
&\le \frac{1}{\rho_x^{\la}}\sum_{k=0}^K \sum_{j=1}^{N_k} \int_{B_{k,j}} |f|^p\,dy
\le \frac{1}{\rho_x^{\la}}\sum_{k=0}^K \sum_{j=1}^{N_k}  ( 2^k R)^{\al \la}
\\
&\lec \frac{1}{\rho^{\al \la}}\sum_{k=0}^K (2^k R )^{\al \la+3 -3\al}
\approx \frac{1}{\rho^{\al \la}}(\rho^{\al} )^{\al \la+3 -3\al}\approx \rho ^{\al(1-\al)(3-\la)}.
}
As $R+\rho^\al \sim \rho^\al$, this bound is the same as \eqref{0920}.

We have shown the second inequality of \eqref{lem2.1} with
$C_2 = C \bke{\frac \rho{R+\rho^\al}}^{\al(3-\la)}$.
\end{proof}

\emph{Remark.} In particular, $C_2 \lec \rho ^{\al(1-\al)(3-\la)}$.
If $\al=1$, we have $C_2 \lec 1$. It is consistent with \cite[Lemma 2.1]{BKT}.
As a result, for $\al \in [0,1]$,
\[
\rho^{\al(\la-2)} \|u_0\|^2_{F^{2}_{\al,\lambda}(\rho)} \lec \rho^{\al(\la-2)}  \rho ^{\al(1-\al)(3-\la)}
= (\rho^\al)^{1-3\al+\al\la}. %
\]
Hence the vanishing condition
\eqref{vanishing} is automatic when $1-3\al+\al\la<0$,
 and has to be assumed otherwise, for global existence.

\begin{lem}[Inclusion]\label{inclusion}
Let $1\le p<\infty$, $0 \le \al \le 1$ and $0 \le \la,\mu < 3$.%
We have

(a)
$ L^p(\R^3)\subset F^p_{\al,\la}$, and $ L^p(\R^3)\subset F^q_{\al,\mu}$ if $1-\frac \mu 3 \le \frac qp \le 1$,

(b) $F^p_{\al,\la}\subset F^p_{\al,\mu}$ if $\la<\mu$, and

(c) $F^p_{\al,\la}\subset F^q_{\al,\mu}$ if {$q<p$} and $\frac{3-\mu}q = \frac{3-\la}p$.
\end{lem}

\begin{proof}
First, we prove part (a). Let $u\in L^p(\R^3)$. For any $x\in \R^3$ and $R\ge 1$ we have
\[
\frac{1}{R_x^{\la}}\int_{B(x,R_x)}|u(y)|^p~dy\le \|u\|^p_{p,\R^3}
\]
where $R_x=\frac18\{R,|x|\}^\al$.
Therefore, $u\in F^p_{\al,\la}$. The second half of (a) follows from (c) by taking $\la=0$.
Part (b) follows from basic algebraic inequality, for all $0\le\la<\mu$ and $x\in\R^3,~R\ge 1$:
\[
\frac{1}{R_x^{\mu}}\int_{B(x,R_x)}|u(y)|^p~dy\le \frac{1}{R_x^{\la}}\int_{B(x,R_x)}|u(y)|^p~dy.
\]
For part (c) we use H\"older inequality: take any $R\ge 1,~x\in\R^3$ and $\la,\mu,p,q$ satisfying $\frac{3-\mu}q = \frac{3-\la}p$, then
\EQN{
\Big(\frac{1}{R_x^{\mu}}\int_{B(x,R_x)}|u(y)|^q~dy\Big)^{\frac{1}{q}}
&\le \frac{|B_{R_x}(x)|^{\frac{p-q}{pq}}R_x^{\la/q}}{R_x^{\mu/q}}\Big(\frac{1}{R_x^{\la}}\int_{B(x,R_x)}|u(y)|^p~dy\Big)^{\frac{1}{p}}
\\
&=C\Big(\frac{1}{R_x^{\la}}\int_{B(x,R_x)}|u(y)|^p~dy\Big)^{\frac{1}{p}},
}
where $C$ is independent of $R_x$.
\end{proof}

The following lemma is useful in localizing a divergence free vector field in a ball.

\begin{lem}\label{2.4} Fix a radial cut-off function $\chi \in C^\infty_c(\R^3)$ with $\chi(x)=1$ for $|x|<\frac 34$ and $\chi(x)=0$ for $|x|>\frac 78$. Let $\chi_R(x)=\chi(x/R)$.
For any $v \in L^p(B_R)$ with $\div v=0$, $1\le p<\infty$, $R>4$,
we have $w\in W^{1,p}(\R^3;\R^3)$ such that $w(x)=0$ for $|x|>R$,
\EQ{\label{wR.eq1}
\div w= f:= v \cdot \nb \chi_R, \qquad
R\norm{\nb w}_{L^p} +\norm{ w}_{L^p} \le C_p \norm{v\cdot \mathbbm 1_{R}}_{L^p},
}
where $\mathbbm 1_{R}$ is the characteristic function of the support of $\nb\chi_R$.
Moreover, for any $0 \le \al \le 1$, $0 \le \la \le 2$, we have
\EQ{\label{wR.eq2}
\norm{w}_{F^p_{\al,\la}} \le C_{p,\al,\la} \norm{v\cdot \mathbbm 1_{R}}_{F^p_{\al,\la}} .
}
Above, the constants $C_p$ and $C_{p,\al,\la}$ are independent of $R$ and $v$.
\end{lem}
\begin{proof} Let $\Phi\in C^\infty_c(\R^3)$ be supported in $|x|\le \frac 12$ with
$\int_{\R^3} \Phi=1$. Let $\phi_R(x) =R^{-3} \Phi(\frac xR)$.
We use the Bogovskii formula for the ball $B_{R}$:
\EQ{
w(x)=\int_{B_R}  N(x,y)f(y)\,dy, \quad N(x,y)=\frac {x-y}{|x-y|^3} \int_{|x-y|}^\infty \phi_R\bke{y+r \frac {x-y}{|x-y|}}r^2 dr.
}
By the properties of the Bogovskii formula, we have $w \in W^{1,p}_0(B_R)$, $\div w= f$ and $\norm{\nb w}_{L^p(B_R)} \le C \norm{f}_{L^p(B_R)}$ with $C$ independent of $R$. It shows \eqref{wR.eq1}.

We now consider
$ \norm{w}_{F^p_{\al,\la}}$. For $x,y\in B_{cR}$, we have
\[
|N(x,y)| \lec \frac {1}{|x-y|^2} \int_{|x-y|}^{cR} R^{-3} r^2 dr\lec \frac {1}{|x-y|^2}.
\]
Also note the support of $f$ is in $A_R=B_R \setminus B_{R/2}$, and for $|k| \sim R$, we have $r_k \sim R^\al$ and
\EQS{
\int_{B(k,r_k)}|f(y)|\,dy &\le \bke{ r_k^{-\la} \int_{B(k,r_k)}|f(y)|^p\,dy}^{\frac 1p} r_k^{\la/p}   \bke{\int_{B(k,r_k)}1\,dz}^{1-\frac 1p}
\\
&
\lec \frac {\de}R  r_k ^{3-\frac{3-\la}p}, \quad \de:=  \norm{v\cdot \mathbbm 1_R}_{F^p_{\al,\la}} .
}

Fix $x_0 \in \R^3$ and $r=r_{x_0}$.
If $B(x_0,2r)$ intersects $A_R$, for any $x \in B(x_0,r)$, decompose
\EQ{\label{w.dec}
w(x) = \int_{A_R\cap B(x_0,2r)}  N(x,y)f(y)\,dy+\int_{A_R\setminus  B(x_0,2r) }  N(x,y)f(y)\,dy=: w_1+w_2.
}
By potential estimate,
\[
r^{-\la/p} \norm{w_1}_{L^p(B(x_0,r))} \lec r^{-\la/p}\, r \norm{f}_{L^p(B(x_0,2r))}
\lec \frac {r\de}R.
\]
For $w_2$, choosing one $l_0\in \ZZ^3$ closest to $x_0/r$, we have
\begin{align}
\nonumber
|w_2(x)|
&\lec \int_{A_R\setminus  B(x_0,2r) }  \frac  {|f(y)|}{|x-y|^2}\,dy
\lec \sum_{k \in r \ZZ^3 \cap A_R}\int_{B(k,r_k)} \frac  {|f(y)|}{r^2+|k-x_0|^2}\,dy
\\
&\lec  \sum_{l \in  \ZZ^3 \cap A_{R/r}}\frac  {1}{r^2(1+|l-l_0|^2)} \int_{B(k,r_k)}|f(y)|\,dy \qquad (k=rl)
\nonumber
\\
&\lec \sum_{l \in  \ZZ^3 \cap A_{R/r}}\frac  {1}{r^2(1+|l-l_0|^2)} \frac {\de}R  r ^{3-\frac{3-\la}p}
\lec \sum_{l \in  \ZZ^3 \cap A_{2R/r}}\frac  {1}{r^2(1+|l|^2)} \frac {\de}R  r ^{3-\frac{3-\la}p}
\nonumber
\\
&\lec \frac R{r^3} \,\frac {\de}R  r ^{3-\frac{3-\la}p} = \de (R^\al) ^{-\frac{3-\la}p}.
\label{w2.est}
\end{align}
Thus
\[
\bke{r^{-\la} \int_{B(x_0,r)} |w_2|^p }^{1/p} \lec (R^\al) ^{\frac{3-\la}p}\de (R^\al) ^{-\frac{3-\la}p}
= \de .
\]
Summing the estimates for $w_1$ and $w_2$, we get
\[
r^{-\la/p} \norm{w}_{L^p(B(x_0,r))} \lec \frac {r\de}R+ \de \lec \de.
\]

For $B(x_0,2r)$ not intersecting $A_R$, $w(x)$ is nonzero only if $|x_0|<R/2$. In this case, in the decomposition $w=w_1+w_2$ in \eqref{w.dec}, we have $w_1=0$, and the same upper bound \eqref{w2.est} for $w_2$ holds.  Thus
\[
\bke{r^{-\la} \int_{B(x_0,r)} |w|^p }^{1/p} \lec r ^{\frac{3-\la}p}\de (R^\al) ^{-\frac{3-\la}p}
\lec \de.
\]
Taking sup in $x_0\in \R^3$, we get \eqref{wR.eq2}.
\end{proof}

In Theorem \ref{mainthm} we have used the spaces $\mathring{F}^p_{\al,\la}$ and $\mathring{F}^p_{\al,\la,\si}$, defined in Definition \ref{Fpalla.def}, as
the closure of $C^\infty_c(\R^3)$ and $C^\infty_{c,\si}(\R^3)=\{ \zeta\in C^\infty_c(\R^3;\R^3), \div \zeta=0\}$ in $F^p_{\al,\la}$-norm. They are needed for the initial data to get local existence. The following lemma will clarify the relation between them and the subspaces
$\hat{F}^p_{\al,\la}$ and $\hat{F}^p_{\al,\la,\si}$ for functions and vector fields vanishing at spatial infinity:
\EQN{
\hat{F}^p_{\al,\la}(R)&=\bigg\{f\in F^p_{\al,\la},\quad \lim_{|x|\to \infty} \frac{1}{r_x^\la}\int_{B(x,r_x)}|f|^p=0\bigg\},
\\
\hat{F}^p_{\al,\la,\sigma}(R)&=\bigg\{f\in F^p_{\al,\la,\si},\ \lim_{|x|\to \infty} \frac{1}{r_x^\la}\int_{B(x,r_x)}|f|^p=0\bigg\},
}
where $r_x=\frac18\max\{|x|,R\}^{\al}$.
By Lemma \ref
{vanishing.initial.data},
 we may assume $R=1$ and omit the dependence on $R$, using the notations $\hat{F}^p_{\al,\la}=\hat{F}^p_{\al,\la}(1)$ and $\hat{F}^p_{\al,\la,\sigma}=\hat{F}^p_{\al,\la,\sigma}(1)$.%

Now we can prove the following lemma that connects spaces $\mathring{F}$ and $\hat{F}$:

\begin{lem}\label{F_hatspaces}
Let $1\le p<\infty$, $0 \le \al \le 1$ and $0 \le \la < 3$.
We have

\textup{(a)}\   $\mathring{F}^p_{\al,\la}=\hat{F}^p_{\al,\la}$.

\textup{(b)}\ $\mathring{F}^p_{\al,\la,\sigma}=\hat{F}^p_{\al,\la,\sigma}$.

\end{lem}
\begin{proof}
  First we will prove that $\mathring{F}^p_{\al,\la}\subset\hat{F}^p_{\al,\la}$. Take $u\in \mathring{F}^p_{\al,\la}$ and let $\varepsilon>0$, then there exists $u_{\varepsilon}\in C^{\infty}_c$ such that $\|u-u_{\varepsilon}\|_{F^p_{\al,\la}}\le\varepsilon$. Since $u_{\varepsilon}$ is compactly supported there exists $\rho$ and $\supp\{u_{\varepsilon}\}\subset B_{\rho/2}$, then we have the following bound
  \EQ{
  \underset{|x|\ge\rho}\sup\bke{\frac{1}{R_x^\la}\int_{B(x,R_x)}|u|^p~dx}^{1/p}\le  \|u-u_{\varepsilon}\|_{F^p_{\al,\la}}\le\varepsilon
  }
Therefore $\mathring{F}^p_{\al,\la}\subset\hat{F}^p_{\al,\la}$. Proof of inclusion $\mathring{F}^p_{\al,\la,\sigma}\subset\hat{F}^p_{\al,\la,\sigma}$ follows the same argument with additional assumptions that $u,u_{\varepsilon}$ are divergence free.

Next we will prove the other inclusion $\hat{F}^p_{\al,\la,\sigma}\subset\mathring{F}^p_{\al,\la,\sigma}$.
  Take $v\in \hat{F}^p_{\al,\la,\sigma}$ then for any $\varepsilon>0$ there exists $\rho$ such that
   \EQ{\label{2.13}
  \underset{|x|\ge \rho}\sup\bke{\frac{1}{R_x^\la}\int_{B(x,R_x)}|v|^p~dx}^{1/p}\le\varepsilon.
  }
  We apply Lemma \ref{2.4} with $R=2\rho$ to find $w$ satisfying \eqref{wR.eq1} and \eqref{wR.eq2}.
Extend $w$ by zero to the whole space $\R^3$. We choose $v_{\varepsilon}=v\chi_R-w$, where $\chi_R$ is the same cut-off function as in Lemma \ref{2.4}. From \eqref{wR.eq1} we get that $\div v_{\varepsilon}=0$ and
\[
\|v-v_{\varepsilon}\|_{F^p_{\al,\la}}\le  \|v(1-\chi_R)\|_{F^p_{\al,\la}} +   \|w\|_{F^p_{\al,\la}}
\le \e + C_{p,\al,\la}\varepsilon.
\]
Here we used \eqref{2.13} and \eqref{wR.eq2} for the last inequality. Take a suitable mollification $v_\e * \eta_t \in C^\infty_{c,\si}$ with $\norm{v_\e -v_\e * \eta_t }_{L^p} \le \e$, we have $\|v-v_{\varepsilon}* \eta_t \|_{F^p_{\al,\la}}\le C\e$. The inclusion
$\hat{F}^p_{\al,\la}\subset\mathring{F}^p_{\al,\la}$ is similar but without the need of $w$, and is skipped.
\end{proof}

\begin{remark}
In case $\al=1$, part (b) of Lemma \ref{F_hatspaces} has an alternative proof, similar to \cite{BK, BKT}, using a Bogovskii map in the annulus $A_R$ instead of the ball $B_R$. Specifically, we have
\[
w\in W^{1,p}_0(A_R),\quad \div w=v\cdot \nb \chi_R,\quad
\norm{\nb w}_{L^p(A_R)}  \le C_p \norm{v\cdot \nb \chi_R}_{L^p}.
\]
As $A(R)$ can be covered by $N$ balls $B_k=B(x_k,r_{x_k})$, $1\le k\le N$, with $N$ independent of $R$,
if $B(x,r_x)$ intersects the support of $w$, then $r_x \sim R$ and
\EQN{
\bke{\frac1{r_x^\la} \int_{B(x,r_x)} |w|^p}^{\frac1p}
&\lec R^{-\frac\la p} R  \bke{\int_{A(R)} |v\cdot \nb \chi_R|^p}^{\frac1p}
\lec R^{-\frac\la p} \bke{ \sum_{k}\int_{B_k} |v|^p}^{\frac1p}
\\
&\lec R^{-\frac\la p} \bke{ \sum_{k} R^\la o(1)}^{\frac1p} = o(1)
}
where $o(1) \to 0$ as $R\to \infty$. We then define $u^R= f \chi_R - w$ and smoothrize it.

This argument
does not work for $\al<1$ mainly due to the differences in annulus covering Lemma \ref{annulas-covering} for $\al=1$ and $\al<1$. Specifically, we will need $N(R) =O(R^{3-3\al})$ balls $B_k= B(x_k, r_{x_k})$ to cover $A_R$, with $N(R) \to \infty$ as $R \to \infty$ if $\al<1$.\hfill $\square$
\end{remark}

\medskip
\begin{remark}
When $\al=0$, $F^p_{0,\la} = L^p_{\uloc}$. In this case, part (b) of Lemma \ref{F_hatspaces} is assumed in \cite{LR} and proved in \cite[Appendix]{KiSe}. However, the proof of \cite{KiSe} has a gap: It claims Lemma \ref{2.4} by applying Bogovskii map to the unit ball $B_k=B(k,1)$ for each $k \in \ZZ^3$ and get
$v^k \in W^{1,p}_0(B_k;\R^3)$ for each $k\in \NN$ that is a solution of the equation%
\[
\div v^k = (f\cdot \nb \chi_R)\ph_k - c_k\quad \text{in }B_k, \quad c_k:=\frac 1{|B_k|}\int_{B_k} (f\cdot \nb \chi_R)\ph_k\,dx,
\]
where $\ph_k$ is a partition of unity such that
\[
\ph_k \in C^\infty_c(\R^3), \quad \text{spt } \ph_k \subset B_k, \quad \ph_k \ge 0, \quad \textstyle \sum_k \ph_k=1.
\]
Extending $v^k$ by zero to the whole space $\R^3$ and setting $v^R = \sum_k v^k$, we get
\[
\div v^R = \sum_k (f\cdot \nb \chi_R)\ph_k -  \sum_k c_k \mathbbm 1_{B_k}= f\cdot \nb \chi_R -\sum_k c_k\mathbbm 1_{B_k}.
\]
The proof of \cite{KiSe} overlooked the factor $\mathbbm 1_{B_k}$ and thought the last sum equaled $\sum_k c_k=0$ so that $v^R$ was our desired $w$. It is unclear how to fix this gap, and Lemma \ref{F_hatspaces} provides a proof.
\hfill $\square$
\end{remark}

\begin{example}\label{exampleS2a}
Let $\chi\in C^1_c(B_1)$ be a bump function.
For $k \in \NN$, let $x_k = 2^ke_1$, $1\le R_k\le 2^{k\al}$, $|a_k|\le 1$, and
\[
f(x) =\sum_{k\in \NN} a_k 2^{k \al \la} R_k^{-3}\,\chi\bke{\frac{x-x_k}{R_k}}.
\]
Then $f\in F^1_{\al, \la}$. If $a_k=o(1)$ as $k \to \infty$, then $f\in \mathring F^1_{\al, \la}$. If we take $a_k=R_k=1$, then $f \not \in L^2_\uloc$ if $\al>0$.
\end{example}

\begin{example}\label{exampleS2b}
Let
\[
f(x)= \frac 1{(1+|x|)^{m}} , \quad 0\le m<3.
\]
(If $m\ge3$, then $f $ is localized and not so interesting.)
If  $|x|>R$, then
\[
\frac1{r_x^{\la}} \int_{B(x,r_x)} f \approx |x|^{\al(3-\la)-m} .
\]
Its sup over $|x|>R$ is finite if $\al(3-\la) \le m$.
If  $|x|<R$, then
\[
\frac1{r_x^{\la}} \int_{B(x,r_x)} f \lec R^{-\al\la} (R^{\al})^{3-m}=R^{\al(3-\la-m)} .
\]
Thus
\[
\norm{f}_{F^1_{\al,\la}(R)} \sim R^{\al(3-\la-m)} , \quad  \text{if }\al(3-\la)\le m ,
\]
and $f\in \mathring F^1_{\al,\la}$ if $\al(3-\la)< m$. However, if $\al(3-\la)< m<3-\la$, (for which we need $\al<1$), then
$f\in \mathring F^1_{\al,\la}$, $f\not \in F^1_{\al=1,\la}$, and $\lim _{R \to \infty} \norm{f}_{F^1_{\al,\la}(R)}=\infty$. Moreover,
\[
R^{\al(\la-2)} \norm{f}_{F^1_{\al,\la}(R)}\sim R^{\al(1-m)}
\]
still goes to infinity as $R \to \infty$ if $m  <1$. This shows that, if $f=|u_0|^2$, then $T_n$ defined by \eqref{Tn.def} may not go to $\infty$ if
$u_0\in \mathring F^2_{\al,\la}$ with $\al(3-\la)< 1$ by choosing $m \in (\al(3-\la), 1)$.
This example does not exist when $\al=1$ since $m$ cannot be chosen.\hfill $\square$
\end{example}

The following lemma is concerned with the borderline case $(3-\la)\al=1$ of \eqref{vanishing}.
We will let $f=|u_0|^2$, so $\norm{f}_{F^{1}_{\al,\lambda}(\rho)} =\norm{u_0}_{F^{2}_{\al,\lambda}(\rho)}^2$, to show that $u_0 \in \mathring F^2_{\al,\la}$ implies \eqref{vanishing}. This result is already known for $(\al,\la)=(1,2)$ by \cite[Lemma 2.1]{BKT}.

\begin{lem}\label{lem2.9}
Let $0 \le \al \le 1$, $0 \le \la \le 2$, and $(3-\la)\al=1$. Suppose $f \in \mathring F^1_{\al,\la}$, then
\[
\lim_{\rho\to \infty} \rho^{\al(\la-2)} \|f\|_{F^{1}_{\al,\lambda}(\rho)} =0.
\]
\end{lem}
\begin{proof}
We may assume $f\ge 0$. Note
\[
\rho^{\al(\la-2)} \|f\|_{F^{1}_{\al,\lambda}(\rho)} = 8^\la \max(I,J),
\]
where
\[
I= \sup_{|x|<\rho} \rho^{-2\al} \int_{B(x,\frac18 \rho^\al)} f,\qquad
J= \sup_{|x|\ge\rho} \rho^{\al(\la-2)} |x|^{-\al \la} \int_{B(x,\frac18 |x|^\al)} f.
\]

Since $f \in \mathring F^1_{\al,\la}$,
for any $\e>0$, there is $R>1$ such that
\EQ{\label{1018}
\sup_{|x|>R/2} |x|^{-\al \la} \int_{B(x,\frac18 |x|^\al)} f \le \e.
}
Consider $\rho \gg R$ so that $\rho^\al > 8R$. As $ \rho^{\al(\la-2)}\le1$, we have $J \le \e$.

We next focus on $I$. Denote $R_x = \frac18\max(R,|x|)^\al$ and $\rho_x = \frac18\rho^\al$ for $|x|<\rho$.

If $2\rho_x < |x| < \rho$, all $y \in B(x,\rho_x)$ has $|y| \sim |x|$ and hence $R_y \sim R_x=\frac 18|x|^\al$ since $R< \rho_x<|x|$. Hence $B(x,\rho_x)$ can be covered by $K$ balls $B(y_k,R_{y_k})$ with $K \lec (\rho_x/R_x)^3\approx (\rho/|x|)^{3\al}$, and, using \eqref{1018},
\EQN{
\frac{1}{\rho^{2\al}}\int_{B(x,\rho_x)}f
&\le \frac{1}{\rho^{2\al}}\sum_{k=1}^K \int_{B(y_k, R_{y_k})}f
\\
&\lec \frac{1}{\rho^{2\al}}\sum_{k=1}^K R_{y_k}^\la \e
\lec \frac{1}{\rho^{2\al}} (\rho/|x|)^{3\al} |x|^{\al\la} \e
\\
& = \rho ^\al |x| ^{-(3-\la)\al} \e
= \rho ^\al |x| ^{-1} \e \le C \e.
}
The second equality uses $\al(3-\la)=1$. The last inequality uses $2\rho_x < |x|$.

\medskip
If $|x|<2\rho_x$, then $B(x,\rho_x)\subset B(0,\rho^\al)$. We cover $B(0,\rho^\al)$ by $B_R(0)$ and $K$ annulus
\[
\mathcal A_k = \{ y \in \R^3: 2^{k-1} R \le |y|< 2^{k}R\}, \quad k=1,\ldots,K,
\]
with $2^{K-1} R \le \rho^\al< 2^{K}R$. Each $\mathcal A_k$ can be covered by $N_k$ balls $B_{k,j}=B(y_{k,j}, R_{y_{k,j}})$ with $y_{k,j} \in \mathcal A_k$, $R_{y_{k,j}} \sim ( 2^k R)^\al $, and $N_k \lec ( 2^k R)^{3(1-\al)} $. Denoting $C_R=\int_{B_R(0)} f$ and using \eqref{1018},
\EQN{
\frac{1}{\rho^{2\al}}\int_{B(x,\rho_x)} f
&\le \frac{1}{\rho^{2\al }}\bigg( C_R
+ \sum_{k=1}^K \sum_{j=1}^{N_k} \int_{B_{k,j}} f \bigg)
\\
&\le \frac{C_R}{\rho^{2\al }}
+ \frac{C}{\rho^{2\al }}\sum_{k=1}^K  ( 2^k R)^{3(1-\al)+\al \la} \e
\\
&
\le  \frac{C_R}{\rho^{2\al }}+
 \frac{C}{\rho^{2\al }}(\rho^{\al} )^{3 -3\al+\al \la}\e
 = \frac{C_R}{\rho^{2\al }} +C
 \e.
}
The last equality uses $\al(3-\la)=1$. The first term is bounded by $\e$ when $\rho$ is sufficiently large, using $\al >0$.

Combining the 2 cases we have
shown $I \le C \e$ for $\rho$ sufficiently large. Since $\e>0$ is arbitrary, we have shown the lemma.
\end{proof}

We include a second proof of Lemma \ref{lem2.9} which uses an auxiliary function $g$ to avoid the series appearing in the preceding proof. 

\begin{proof}[Second proof of Lemma \ref{lem2.9}]
Let $g(x)= (1+|x|^{3\al -\al \la})^{-1}$
and $r_x = \frac18\max(\rho,|x|)^\al$. {For the moment we do not assume
$3\al -\al \la=1$.}
We claim there exists a constant $\td c>0$ independent of $\rho$ so that
\EQ{\label{gintinf}
\td c^{-1}\leq \inf_{x\in \R^3}%
\frac 1 {r_x^\la}\int_{B(x,r_x)} g.
}
Indeed, %
\[
\frac 1 {r_x^\la}\int_{B(x,r_x)} g 
\ge
 r_x^{3-\la} \min_{B(x,r_x)} g =\frac{
 r_x^{3-\la}}{ 1+(|x|+r_x)^{(3-\la)\al}  }\gec
 \frac{
 r_x^{3-\la}}{ 1+\max(\rho,|x|)^{(3-\la)\al}  }\approx 1.
\]
It is bounded from below by a constant independent of $x$ and $\rho$.
Observe also that, for $\rho\geq 1$, %
and under the requirement that $3\al - \al\la=1$,%
\[
\frac 1 {\rho^{2\al}} \int_{B(0,3\rho^\al)} g< \td C,
\]
where $\td C$ is independent of $\rho$.

For  $|x|\leq \rho$, let $\mathcal C$ be a covering of  {$B(x,\rho^\al/8)$}
by balls of the form $B(y,r_y)$  which are taken so that there is a uniform bound on the number of cover elements that intersect any single cover element.  The following bound is the key tool in this proof,%
\EQ{\label{ineq.sum}
\sum_{x_k\in \mathcal C} \frac {r_{x_k}^\la} {\rho^{2\al}} \leq  \sum_{x_k\in \mathcal C} \frac {\td c} {\rho^{2\al}} \int_{B(x_k,r_{x_k})} g \lesssim   \frac c{\rho^{2\al}} \int_{B(x,3\rho^{\al})}  g \lesssim   \frac c{\rho^{2\al}} \int_{B(0,3\rho^{\al})}  g   <C,
}
where, abusing notation, we are indexing cover elements by their centers $x_k$.
Note that in the second last inequality we were able to pass from $B(x,3\rho^\al)$ to $B(0,3\rho^\al)$ because $g$ is radial and non-increasing in $|x|$.
The uniform bound above is independent of $\rho$ and $x$.%

Suppose $f\in \mathring F^1_{\al,\la}$ and, without loss of generality,
$\|f\|_{F^{1}_{\al,\la}}=1$.
 Let
\[
\phi(r) = \sup_{|x|\geq r} \frac 1 {r_{x}^\la} \int_{B(x,r_x)} f.
\]
Then $\phi$ is decreasing in $r$ and bounded by $1$. Suppose $|x|\leq \rho$. 
Let $\mathcal C$ be a covering of $B(x,\rho^\al/8)$ as above.
Then%
\EQ{
\frac 1 {\rho^{2\al}} \int_{B(x,\rho^\al/8)} f 
\lesssim \sum_{x_k\in \mathcal C} \frac 1 {\rho^{2\al}}\int_{B(x_k,r_{x_k})} f
\le \sum_{x_k\in \mathcal C} \phi(|x_k|)\frac {r_{x_k}^\la} {\rho^{2\al}} .
}
 Let $\e>0$ be given.  The above is equal to %
\EQN{
&=  \sum_{x_k\in \mathcal C;|x_k|\leq \e\rho^{\al}} \phi(|x_k|)\frac {r_{x_k}^\la} {\rho^{2\al}} + \sum_{x_k\in \mathcal C;|x_k|> \e\rho^{\al}} \phi(|x_k|) \frac {r_{x_k}^\la} {\rho^{2\al}}
\\&\leq    \sum_{x_k\in \mathcal C;|x_k|\leq \e\rho^{\al}}  \e^2  \frac {r_{x_k}^\la} {\e^2 \rho^{2\al }}   + \phi(\e\rho^\al)
\lesssim \e^2 	+\phi(\e\rho^\al),
}
where we have used \eqref{ineq.sum} applied to the ball $B(0,2\e \rho^\al)$ to get the $\e^2$ term on the right-hand side, noting that the collection of balls with $|x_k|\leq \e\rho^\al$ is contained in a suitable cover of $B(0,2\e \rho^\al)$. With $\e$ fixed, by taking $\rho$ large $\phi(\e\rho^\al)$ can be made small.

To conclude the proof note that
\[
\sup_{|x|\geq \rho}  \rho^{\al\la -2\al} \frac 1 {r_x^\la}\int_{B(x,r_x)}f\leq  \rho^{\al\la-2\al} \phi( \rho)  \to 0,
\]
as $\rho\to \I$.
\end{proof}

{The following lemma will be useful to Lemma \ref{Remark.ER} and Remark \ref{rem2.11}. It is
given in \cite[Lemma 2.1]{BKT} for $M^{2,2}$, but the result for general $M^{p,\la}$, $\la>0$, is proven identically.} 
 
\begin{lem}[Centered definition]\label{Mpq-centered} 
Let $p ,R\ge 1$ and $\la >0$.
The Herz space $M^{p,\la}=F^p_{1,\la}$ defined in \eqref{Mpq} has three equivalent norms: 
\[
\|f\|_{M^{p,\la}(R)}^p =\frac1{R^\la}\!\int_{B_R(0)}|f|^p +
\sup_{|x|>R}  \frac1{|x|^\la}\!\int_{B(x,  \frac{|x|}2)}|f|^p,
\quad
\|f\|_{ M_*^{p,\la}(R)}^p = \sup_{r\ge R}\, \frac 1 {r^\la}\! \int_{B(0,r)}|f|^p\,dx,
\]
and
$\|f\|_{F^{p}_{1,\la}(R)}$.
The constant $c_1\in (0,1)$ with 
$\|f\|_{M_*^{p,\la}(R)}, \|f\|_{F^{p}_{1,\la}(R)}  \in (c_1, c_1^{-1}) \|f\|_{M^{p,\la}(R)}$ 
is independent of $R \ge 1$.
For $f\in M^{p,\la}$, the following 3 statements are equivalent:\smallskip

\textup{(a)} $f\in \mathring M^{p,\la}$ defined in \eqref{Mpqring}.

\textup{(b)} $\lim_{R\to \infty} \|f\|_{M^{p,\la}(R)}=0$.

\textup{(c)} $\lim_{r\to \infty} \frac 1 {r^\la}\!\int_{B(0,r)}|f|^p\,dx=0$.

\end{lem}

\begin{proof} [Sketch of proof]  {For the equivalence of norms, since the proof is similar to that in  \cite[Lemma 2.1]{BKT}, we only show as an example that}
\[
\| f\|_{{ F}^{p}_{1,\lambda}(R)}^p\sim  \sup_{r>R} \frac 1 {r^\lambda}\int_{B(0,r)}|f|^p\,dx.
\]
By Lemma \ref{annulas-covering} when $\al=1$, shells of the form $B(0,2r)\setminus B(0,r)$ and  the ball $B(0,R)$, can be covered by a  bounded number {(independent of $r,R$)} of elements of $\{ B(x,r_x)\}$ where $r_x = \frac 1 8 \max\{R,|x| \}$ and $R \le r$. 
This implies, for instance, that, taking $R\leq r $ and  {$k\ge0$ so that $2^kR\leq r <2^{k+1}R$,}
\[
\frac 1 {r^\lambda} \int_{B(0,r)} |f|^p \,dx\lesssim  \frac 1 {r^\lambda} \sum_{i=0}^{k} (2^{i}R)^\lambda \frac {1}  {(2^{i}R)^\lambda}\! \int_{2^i R \leq |x|\leq 2^{i+1}R}|f|^p\,dx				+  \frac 1 {R^\lambda}\!\int_{B(0,R)} |f|^p\,dx\lesssim \| f\|_{{ F}^{p}_{1,\lambda}(R)}^p,
\]
where we used a finite cover of each shell or ball to pass from the centered integrals to those over elements of $\{ B(x,r_x)\}$. The reverse direction of the norm equivalence is straightforward.

For the equivalence of vanishing statements, (b) and (c) are clearly equivalent and imply (a) by the equivalence of 3 norms. Suppose (a) holds. For any $\e>0$ there is $R>1$ so that  $\sup_{|x|>R}  \frac1{|x|^\la}\!\int_{B(x,  \frac{|x|}2)}|f|^p< \e$. For $r\gg R$, $2^kR\leq r <2^{k+1}R$, we have
\EQN{
\frac 1 {r^\lambda} \int_{B(0,r)} |f|^p \,dx &\le  \frac 1 {r^\lambda} \sum_{i=0}^{k} (2^{i}R)^\lambda \frac {1}  {(2^{i}R)^\lambda}\! \int_{2^i R \leq |x|\leq 2^{i+1}R}|f|^p\,dx				+  \frac 1 {r^\lambda}\!\int_{B(0,R)} |f|^p\,dx
\\
&\le \frac 1 {r^\lambda} \sum_{i=0}^{k-1} (2^{i}R)^\lambda C\e
+  \frac 1 {r^\lambda}C_R
\le C\e +  \frac 1 {r^\lambda}C_R
}
which is less than $C\e$ for $r$ sufficiently large. Hence (a) implies (c).
\end{proof}

The following lemma confirms the embeddings illustrated in Figure \ref{figure.ER}.
{Recall $\mathring {\mathfrak F}^{2}_{\al,\la}$ is defined in \eqref{Ffrak0.def}.}
\begin{lem}\label{Remark.ER}
If $\al_1\la_1=\al_2\la_2=q\in (0,1]$ with $\la_1 < \la_2$ and $0<\al_1,\al_2\leq 1$,
then
\[
\mathring {\mathfrak F}^{2}_{\al_1,\la_1}\subsetneq \mathring {\mathfrak F}^{2}_{\al_2,\la_2}.
\]
\end{lem}

\begin{proof}

Let $\al_1\la_1=\al_2\la_2=q$ with $\la_1 < \la_2$. Then
\EQN{
&\sup_{x\in \R^3} \frac 1 {\max(R ,|x| )^{\al_2\la_2 }}  \int_{B(x, \max(R,|x|)^{\al_2}/8  )} |f|^2\,dx \\\quad &\leq \sup_{x\in \R^3} \frac 1 {\max(R ,|x| )^{\al_1\la_1 }} \int_{B(x, \max(R,|x|)^{\al_1}/8)} |f|^2\,dx.
}
It implies
\[
\| f\|_{F^{2}_{\al_2,\la_2} (R)}\leq \| f\|_{F^{2}_{\al_1,\la_1} (R)},
\]
and so
\[
\mathring {\mathfrak F}^{2}_{\al_1,\la_1}\subset \mathring {\mathfrak F}^{2}_{\al_2,\la_2}.
\]

We now show this inclusion is strict. Let
\[
f(x) = \frac 1 { \sqrt{   |x|^{3- \frac {\la_1} 2 -\frac {\la_2}2 } +1  }   }.
\]
We have that
\[
\frac 1 {R^{\al_1\la_1}} \int_{B(0,R^{\al_1})} |f|^2\,dx \sim R^{\al_1(\la_2-\la_1)/2},
\]
which is not bounded in $R$. Hence, $f\notin   {\mathfrak F}^{2}_{\al_1,\la_1}$.

On the other hand, since $f$ is decreasing in $|x|$,   we have
\[
\sup_{|x|\leq R} \int_{B(x,R^{\al_2})}|f|^2\,dx \leq \int_{B(0,R^{\al_2})} |f|^2\,dx,
\]
and
\[
\frac 1 {R^{\al_2\la_2}} \int_{B(0,R^{\al_2})} |f|^2\,dx \sim R^{\al_2(\la_1-\la_2)/2}.
\]
This vanishes as $R\to \I$.
Additionally note that \EQN{
\frac 1 {|x|^{\al_2 \la_2 }} \int_{B(x,|x|^{\al_2}/8)}|f|^2 \,dx&\lesssim |x|^{3\al_2       -\al_2\la_2 } \bigg(\frac 1 { |x|^{\al_2} (|x|^{1-\al_2} - \frac 1 8)  } \bigg)^{3- \frac {\la_1} 2 - \frac {\la_2} 2 }
\\&=|x|^{  \al_2( \la_1- \la_2)/2}\bigg(	\frac 1 {|x|^{1-\al_2} -\frac 1 8}\bigg)^{3- \frac {\la_1} 2 - \frac {\la_2} 2}.
}
After applying  $\sup_{|x|>R}$ to the right-hand side above we get something that vanishes as $R\to \I$. It follows that $f\in \mathring {\mathfrak F}^{2}_{\al_2,\la_2}$.
\end{proof}

\begin{remark} \label{rem2.11}
In the preceding proof a centered definition of the norm of  $M^{2,q}$, {as in Lemma \ref{Mpq-centered}}, is used.  In this remark we show that there is no  centered definition of $\|\cdot\|_{F^{p}_{\al,\la}}$ when $\al< 1$. We will not use this observation anywhere, but hope that it is illuminating to the reader.
By this we mean that there exists no function $\Phi$ so that
\EQ{\label{eq2.16}
\|f\|_{F^p_{\al,\la}{(R=1)}}^p \sim \sup_{r\geq 1}\frac 1 {\Phi(r)} \int_{B(0,r)}|f|^p\,dx.
}
Without loss of generality we only consider $F^1_{\al,\la}$ {with $p=1$.}
Fix $\al$ and $\la$.
Let $\mathscr C_r$ denote the collection of all covers of $B(0,r)$ by elements of {$\cB_1$}. 
Let
\[
\Psi(r) = \inf_{\bar C\in \mathscr C_r} \sum_{B(x,r_x)\in \bar C} r_x^\la, \quad
{r_x = \frac18 \max(1,|x|)^\al.}
\]
Then, 
\EQ{\label{eq2.17}
\sup_{r\geq 1} \frac 1 {\Psi(r)} \int_{B(0,r)} |f| \lesssim \|f\|_{F^1_{\al,\la}}.
}
We may approximate $\Psi(r)$  using the function $g$ in the second proof of Lemma \ref{lem2.9}, {defined above \eqref{gintinf}.} Note that we may choose $\td c$ {in \eqref{gintinf}} so that
\EQ{\label{1216}
\td c^{-1}\leq \sup_{x\in \R^3}\frac 1 {r_x^\la} \int_{B(x,r_x)}g(y)\,dy\leq \td c.
}
The upper bound is new compared to what is written in the second  proof of Lemma \ref{lem2.9}, but the supporting estimate is similar to the lower bound, {using $R=1$}.
Choosing a particular cover $\mathscr C$ of $B(0,r)$ in which each cover element is overlapped by a bounded number $\bar N$ of other cover elements, we obtain
\[
\Psi(r) \leq \sum_{B(y,r_y) \in \mathscr C} r_{y}^\la \leq \td c \sum_{B(y,r_y) \in \mathscr C}\int_{B(y,r_y)}g(x)\,dx  \leq \td c \bar N\int_{B(0,2r)} g(x)\,dx \sim r^{3-3\al +\al \la}. 
\]
On the other hand, for any given cover $\mathscr C'$ of $B(0,2r)$, we have 
\[
r^{3-3\al +\al \la} \sim \int_{B(0,2r)} g(x)\,dx\leq \sum_{B(y,r_y)\in\mathscr C'}\int_{B(y ,r_y)}  g \,dx \leq 
\td c \sum_{B(y,r_y)\in\mathscr C'}r_y^\la,
\]
{using \eqref{1216} in the last inequality.}
This implies 
\[
\Psi(r) \sim  r^{3-3\al +\al \la}.
\]

\newcommand{\barPhi}{\Phi}
 
We now argue that if $\barPhi(r)$ satisfies 
\EQ{\label{ineq.needed}
\sup_{r\geq 1} \frac 1 {\barPhi(r)} \int_{B(0,r)} |f| \leq \bar c \|f\|_{F^1_{\al,\la}},
} 
for some $\bar c>0$ and all $f\in F^1_{\al,\la}$, 
then $ \Psi(r) \lesssim \barPhi(r)$.   Using the function $g$ again, we have 
\[
 {\Psi(r) \|g\|_{F^1_{\al,\la}}\lesssim }\int_{B(0,r)} g \,dx.
\]
Since $g\in F^1_{\al,\la}$ we have 
\[
\frac 1 {\barPhi(r)}\int_{B(0,r)}g \,dx\leq \bar c \|g\|_{F^1_{\al,\la}} \lesssim\frac 1 {\Psi(r)} \int_{B(0,r)} g \,dx,
\]
implying $\Psi(r)\lesssim \barPhi(r)$. This  {and \eqref{eq2.17} imply for all $f\in F^{1}_{\al,\la}$}
\EQ{\label{ineq.forContradiction}
\sup_{r\geq 1} \frac 1 {\barPhi (r)} \int_{B(0,r)} |f|   \lesssim \sup_{r\geq 1} \frac 1 {  \Psi(r)} \int_{B(0,r)} |f| \leq \|f\|_{F^1_{\al,\la}}.
}

We conclude with an argument by contradiction.  If there exists a weighted definition of the norm of $F^{1}_{\al,\la}$ with weight $\barPhi$ {so that \eqref{eq2.16} holds}, then 
{the left term of \eqref{ineq.forContradiction} is equivalent to $\|f\|_{F^1_{\al,\la}}$,}
and the middle term of \eqref{ineq.forContradiction} therefore provides an equivalent norm to that of $F^{1}_{\al,\la}$. But, the middle term is exactly the norm of $M^{1,3-3\al +\al\la}$ {because of $\Psi(r) \sim  r^{3-3\al +\al \la}$ and Lemma  \ref{Mpq-centered}.} So, if there exists a weighted definition of the norm of $F^{1}_{\al,\la}$, then $F^{1}_{\al,\la}= M^{1,3-3\al +\al\la}$ as normed spaces.
This leads to a contradiction if $\al<1$, as evidenced by the function
\[
h(x)= |x|^{2-3\al +\al\la} \chi_{\{|x - (x_1,0,0)|<1\}}(x) \in M^{1,3-3\al +\al\la}\setminus F^1_{\al,\la}.
\]
{Here the region $|x - (x_1,0,0)|<1$ is the cylinder $x^2_2+x_3^2<1$.}
Hence, there is no centered definition of $\|\cdot\|_{F^1_{\al,\la}}$ when $\al<1$.\hfill $\square$ 
\end{remark}

\section{Pressure estimate}\label{sec3}

In this section we prove the pressure estimate in terms of the global $F^p_{\al,\la}(\R^3)$ norm of $u$ with $p=2$,  %
defined in \eqref{Fal.def}.

We will use the following function (defined in \eqref{aln.def}) in the pressure estimate:
\EQ{\label{aln.def2}
a_n(t)=\norm{u(t)}_{F_{\al,\la}^{2}(R=2^n)}^2 =
\sup_{x\in\R^3} \frac1{r_x^\la}\int_{B(x, r_x)} |u(\cdot,t)|^2 ,\qquad r_x=\frac18\max(2^n,|x|)^{\al}.
}

\begin{lem}[Annulus integral]\label{annulas-integral}
Fix $n \in \NN_0$ and let $a_n(t)$ be defined as \eqref{aln.def2}. For integer $l\ge n$,
we have
\EQ{\label{th:annulas-integral}
 \int_{2^l< |y|<2^{l+1}}   |u_iu_j(y,t)|\,dy
  \le C  2^{l(\al \la +3-3\al)}   a_n(t) ,
}
where $C$ is independent of $n$ and $l$. Moreover, we have
\EQ{\label{th:ball-integral}
\int_{B(0,2^n)}|u_iu_j(y,t)|\,dy
  \le C  2^{n(\al \la +3-3\al)}   a_n(t).
}
\end{lem}

\begin{proof}
By Lemma \ref{annulas-covering}, we can cover $A_{2^l}$ by $N(2^l)$ balls $B_{l,k}$, $k=1,\ldots,N(2^l)$, with center $x_{l,k}$ inside $A_{2^l}$ and radius $r_{x_{l,k}}\sim 2^{l\al}$. Hence the integral in \eqref{th:annulas-integral} is bounded by
\[
\sum_{k=1}^{N(2^l)}   \int_{B_{l,k}}   |u_iu_j(y,t)|\,dy
 \lec  \sum_{k=1}^{N(2^l)}   2^{l\al \la} a_n(t)
\lec    2^{l(3-3\al)}  2^{l\al \la} a_n(t) .
\]
The integral in \eqref{th:ball-integral} is bounded similarly.
\end{proof}
For a given ball $B$ of radius $r$, let $B^*$ and $B^{**}$ denote concentric balls of radii $\frac 98r$ and $\frac 54r$, respectively.
 Recall the set $\cB_n$ of balls defined for $n \in \NN_0$ in \eqref{cCn.def},
\Eq{
\cB_n = \bket{ B(z,r_z): \ z \in \R^3} ,\qquad r_x=\frac18\max(2^n,|x|)^{\al}.
}
\begin{lem}[Pressure estimate]\label{lemma.pressure}
Fix $n\in \NN_0$. %
 Assume $u$ is a local energy solution to \eqref{eq.NSE} on $\R^3\times [0,T)$
with an associated pressure $p$.
Then, for $t\in(0,T)$ and any ball $B\in \cB_n$,
\EQS{
\label{EB09}
&\frac 1 {|B|^{1/3}} \int_0^t \int_{B^*} |p-p_B(s) |^{3/2}\,dx\,ds
\\& \quad
\leq
C \sup_{B'\in \cB_n ; \, B'\cap B^{**}\neq \emptyset} \bke{\frac 1 {|B|^{1/3}} \int_0^t \int_{B'} |u|^3\,dx\,ds}
+ C
 |B|^{\frac 12\la -\frac 56}
\int_0^t  a_n(s)^{3/2} \,ds,
}
where $C$ is independent of $n$ and $B$.
\end{lem}
\emph{Remark.}\quad The factor of the last term in \cite[Lemma 3.1]{BKT} is $ |Q|^{\la/2-5/6} (\log \bka{|Q|^{1/3}/2^n})^{3/2 }$
which agrees with $ |B|^{\frac 12\la -\frac 56}$ when $\al=1$ up to a log factor.
In fact, the  log factor in \cite{BKT} can be avoided, as shown by our proof. The log factor in \cite{BKT}
appears because its proof replaces $|Q'|^{\la/3}$ by $|Q|^{\la/3}$ and count all cubes together. Our proof counts balls in each annulus and then takes the sum. The size of $|Q'|^{\la/3}$ in each annulus is different and hence the sum is a geometric series. The idea of using a convergence (non-geometric) series also helps to improve the $J_{2b}$ estimate \eqref{J2.est} below.

\begin{proof}
For a given ball $B=B(z,r_z)\in\cB_n$, we have the local pressure decomposition for $(x,t) \in B^* \times(0,T)$,
\EQS{
\label{EQ01}
p(x,t)-p_B(t) &=   - \frac 13 |u|^2(x,t) +\pv \int_{ y\in B^{**}}  K_{ij}(x-y)   u_i u_j(y,t)\,dy
\\&\quad
+  \int_{ y\notin B^{**}}  \bkt{ K_{ij}(x-y) - K_{ij}(z - y  )}  u_i u_j(y,t)\,dy
\\
& =: p_\text{near}(x,t)+ p_\text{far}(x,t).
}
For $p_\text{near}$, the Calder\'on-Zygmund inequality implies
\EQN{
\frac 1 {|B|^{1/3}} \int_0^t \int_{B^*} |p_\text{near}|^{3/2}\,dx\,ds
  &\leq
\frac C {|B|^{1/3}} \int_0^t \int_{B^{**}} |u|^3 \,dx\,ds.
}
Since $B^{**}$ can be covered by $N$ balls $B'\subset\cB_n$ with $N$ independent of $n$ and $B$, no matter whether
 $|z|>2^{n+1}$
or not, we have
  \[
       \frac 1 {|B|^{1/3}} \int_0^t \int_{B^{**}} |u|^3 \,dx\,ds
     \leq
      C \sup_{B'\in\cB_n,B'\cap B^{**}\neq \emptyset}\frac 1 {|B|^{1/3}}\int_0^t \int_{B'}|u|^3\,dx\,ds.
   \]

For $p_\text{far}$, we have
\Eq{
|p_\text{far}(x,t)| \lec
 \int_{ y\notin B^{**}}  \frac {r_z}{|z-y|^4}  |u_iu_j(y,t)|\,dy.
}
We first consider $B=B(z,r_z)$ with $|z|>2^{n+1}$. Choose integer $m > n$ so that $2^m < |z| \le 2^{m+1}$.
We have
\EQN{
|p_\text{far}(x,t)| &\lec
\bket{\int_{\textup{I} }+ \int_{ \textup{II} } + \int_{\textup{III} }} \frac {r_z}{|z-y|^4}  |u_iu_j(y,t)|\,dy = J_1+J_2+J_3.
}
where
\EQN{
\textup{I} &= \{ y: |y|<2^{m-1}\}
\\
\textup{II} &= \{ y: 2^{m-1}\le|y|\le 2^{m+2} ,\ y\notin B^{**}\}
\\
\textup{III} &= \{ y: |y|>2^{m+2}\}.
}

For $J_1$, we have
\[
J_1 \lec\frac {r_z}{|z|^4} \int_{\textup{I}}   |u_iu_j(y,t)|\,dy.
\]
If $m \ge n+2$,
\EQN{
J_1 &\lec\frac {r_z}{|z|^4} \bke{\int_{B(0,2^n)} +  \sum_{l=n}^{m-2}  \int_{2^l< |y|<2^{l+1}}}   |u_iu_j(y,t)|\,dy.
}
By Lemma \ref{annulas-integral},
\EQN{
J_1
& \lec 2^{m(\al-4)} \bke{2^{n(3-3\al+\al\la)}  a_n(t) +  \sum_{l=n}^{m-2} 2^{l(\al \la+3-3\al)}  a_n(t)   }
\\
& \lec  2^{m(\al \la -2\al-1)} a_n(t) .
}
If $m=n+1$, %
we do not need the sum and have the same bound.

For $J_3$, by Lemma \ref{annulas-integral},
\EQN{
J_3 &\lec\int_{\textup{III}}   \frac {r_z}{|y|^4}  |u_iu_j(y,t)|\,dy
\\
&\lec  \sum_{l=m+2}^{\infty}  \int_{2^l< |y|<2^{l+1}} \frac {2^{m\al}}{2^{4l}}  |u_iu_j(y,t)|\,dy
\\
& \lec  \sum_{l=m+2}^{\infty} \frac {2^{m\al}}{2^{4l}}   2^{l(\al \la+3-3\al)}    a_n(t) =  \sum_{l=m+2}^{\infty} 2^{m\al }  2^{l(\al \la-1-3\al)} a_n(t)
.
}
The sum converges since
$
\al \la < 1+ 3 \al$
(which is always valid if $\la \le 3$),
and we get
\[
J_3 \lec 2^{m(\al \la -2\al-1)} a_n(t) .
\]

For $J_2$, %
we further divide the region II to two parts:
\[
\II_a = \{ y \in \II, \ |y-z|\ge 2^{m-1}\}, \quad
\II_b = \{ y \in \II, \ r_z<|y-z|<2^{m-1}\},
\]
and decompose $J_2=J_{2a}+J_{2b}$ over these two subregions. The estimate of $J_{2a}$ is the same as the integral over the annulus $A_{2^{m+1}}$ which is part of $J_3$. For $J_{2b}$, we cover $\II_b$ by $K$ annuli
\[
\II_b \subset \cup_{k=1}^K \mathcal{A}_k, \quad  \mathcal{A}_k = \bket{ y: kr_z < |y-z| \le (k+1)r_z},
\]
with $K=  \lceil 2^{m-1}/ r_z \rceil \le C2^{m(1-\al)} $.  Each annulus $ \mathcal{A}_k$ can be covered by $N_k$ balls $B_{k,j}=B(x_{k,j}, r_{x_{k,j}})$, $j=1,\ldots,N_k$ with $N_k \lec k^2$
since $|x_{k,j}-z| \sim kr_z$ and $r_{x_{k,j}} \sim r_z$.

By the definition \eqref{aln.def2} of $a_n(t)$,
we have
\EQS{\label{J2.est}
J_{2b} &\lec \sum_{k=1}^K \sum_{j=1}^{N_k}  \frac {r_z}{(kr_z)^4}  \int_{B_{k,j}} |u_iu_j(y,t)|\,dy
\\
&\lec  \sum_{k=1}^Kk^2  \frac {r_z}{(kr_z)^4} r_z^\la a_n(t)
\lec  r_z^{\la-3} a_n(t).
}
Combining all together the final estimates for $J_1,J_{2a}, J_{2b}$, and $J_3$,
\[
J_1 + J_{2a}+J_3 \lec 2^{m(\al \la -2\al-1)} a_n(t),\quad
J_{2b} \lec r_z^{\la-3} a_n(t).
\]
We see that $J_{2b}$ is the dominant term, and
we conclude, when $|z|>2^{n+1}$,
\EQ{\label{eq2.5}
|p_\text{far}(x,t)| \lec  r_z^{\la-3} a_n(t).
}

\bigskip
For the case $B=B(z,r_z),~|z|\le 2^{n+1}$, we have
\EQN{
|p_\text{far}(x,t)| &\lec
\bket{\int_{\textup{I} }+ \int_{ \textup{II} } }  \frac {r_z}{|z-y|^4}  |u_iu_j(y,t)|\,dy = I_1+I_2,
}
where
\[
\textup{I} = \{ y:~y\notin B^{**}, |y|<2^{n+2}\},\qquad
\textup{II} = \{ y:~ y\notin B^{**},|y|>2^{n+2}\}.
\]
Term $I_2$ is estimated similarly to $J_3$ and has the same bound
\[
I_2 \lec  2^{n(\al \la-1-2\al)} a_n(t).
\]
Term $I_1$ is estimated similarly to $J_2$ and has the same bound \eqref{J2.est}. Thus for all cases, we have estimate \eqref{eq2.5} for $p_\text{far}(x,t)$, and
\EQS{
\frac 1 {|B|^{1/3}} \int_0^t \int_{B^*} |p_\far |^{3/2}\,dx\,ds & \lec
\frac 1 {|B|^{1/3}} \int_0^t |B| \cdot ( r_z ^{( \la -3)} a_n(s))^{3/2} \,ds
\\
&=  |B|^{2/3} |B| ^{\frac 12(\la-3)} \int_0^t a_n(s)^{3/2} \,ds .
}
Combining these estimates yields \eqref{EB09}, and the proof is concluded.
\end{proof}

\section{A priori bounds}\label{sec4}

The next is just \cite[Lemma 3.2]{BKT}.
\begin{lem}\label{lemma.cubic}
Let $u:\R^3\times (0,T)$ and $\e>0$ be given. Then, there exists a constant $C(\e)$ so that,  for any ball $B\subset \R^3$,
\EQN{
\frac 1 {|B|^{1/3}}\int_0^t \int_B |u|^3\,dx\,ds &\leq C(\e) |B|^{\la-4/3}\int_0^t \bigg(  \frac 1 {|B|^{\la/3}} \int_B |u|^2\,dx\bigg)^3 \,ds
\\&+ \e \int_0^t \int_B |\nb u|^2\,dx\,ds
\\&+ C |B|^{\la/2- 5/6} \int_0^t \bigg( \frac 1 {|B|^{\la/3}} \int_B |u|^2\bigg)^{3/2}\,ds.
}
\end{lem}
For a ball $B=B (x,r_x)$ we introduce a cut-off function $\phi_B$ which is supported on $B^*$ and identically $1$ on $B$. All such cut-off can be obtained by scaling and translating a single cut-off, in which case we have $|D^k \phi_B| \lesssim r_x^{-k} \sim |B|^{-k/3}$ where suppressed constants are independent of $x$.

Recall \eqref{aln.def} and \eqref{ben.def} that
\[
a_{n}(t)
= \sup_{B\in \cB_n} \frac 1 {|B|^{\la/3}}\int_B |u(x,t)|^2\,dx, \quad
b_{n}(t)=\sup_{B\in \cB_n} \frac 1 {|B|^{\la/3}}\int_0^t\int_{B}|\nb u(x,s)|^2\,dx\,ds.
\]

\begin{lem}[\textit{A priori} inequalities]\label{a.priori}
 Assume $u_0\in F^2_{\al,\la}(\R^n)$ is divergence free and let $(u,p)$ be a local energy solution with initial data $u_0$ on $\R^3\times(0,T)$. Let it also satisfy the following:
\EQ{
\esssup_{0<s<t}a_n(s)+b_n(t)< \infty, \quad \forall t<T.
}
Then for all $t\in(0,T)$ we have the following a priori estimate
\EQ{\label{grad.est.1}
a_n (t) +b_n(t)
\leq Ca_n(0) +C2^{-2n\al} \int_0^t a_n\,ds
+  C 2^{2n\al(\la-2)}\int_0^t a_n^3 \,ds .
}
\end{lem}

\begin{proof}
Fix $B\in \mathcal B_n$. The local energy inequality reads
\EQN{
&\int |u(t)|^2 \phi_B + 2\int_0^t \int |\nb u|^2 \phi_B \leq  \int |u_0|^2\phi_B
\\&\qquad + \int_0^t\!\int |u|^2 \Delta \phi_B  + \int_0^t\! \int \big( |u|^2 u\cdot \nb \phi_B + 2(p-p_B) u\cdot \nb \phi_B   \big).
}
Observe that
\EQN{
 \int_0^t \int |u|^2\Delta \phi_B\lesssim |B|^{-2/3} \int_0^t \int_{B^*} |u|^2.
}
But, $B^*$ can be covered by $N$ balls from $\mathcal B_n$ with $N$ independent of $n,B$, and so
\EQN{
 \int_0^t \int |u|^2\Delta \phi_B\lesssim N |B|^{-2/3} \sup_{B'\in \mathcal B_n; \,B' \cap B^*\not= \emptyset}  \int_0^t \int_{B'} |u|^2 \lesssim |B|^{(\la-2)/3} \int_0^t a_n.
}
Combining this with Lemmas \ref{lemma.pressure} and \ref{lemma.cubic} we obtain
\EQS{\label{est3.2}
& |B|^{-\la/3} \int |u(t)|^2 \phi_B + 2 |B|^{-\la/3}\int_0^t \int |\nb u|^2 \phi_B
 \\
&\leq |B|^{-\la/3}  \int |u_0|^2\phi_B +C|B|^{-2/3} \int_0^t a_n\,ds
+  C_\e |B|^{2\la/3-4/3 }\int_0^t a_n^3 \,ds
\\&\quad +C |B|^{\la/6- 5/6 }\int_0^t a_n^{3/2}\,ds + C \e b_n(t).
}

By Young's inequality for products,
\EQ{
|B|^{\frac{\la-5}6 } a_n^{3/2}  =  (|B|^{-\frac12 } a_n^{3/4})\cdot (|B|^{\frac{\la-2}6 } a_n^{3/4}) \le
(|B|^{-\frac12 } a_n^{3/4})^{4/3}+ (|B|^{\frac{\la-2}6 } a_n^{3/4})^4.
}
Hence the second last term in  \eqref{est3.2} can be dropped.

Thanks to the assumption $\la \le 2$, all exponents of $|B|$ in \eqref{est3.2} are nonpositive.
Observe that the smallest ball in $\mathcal B_n$ has radius $\sim 2^{n\al}$. Hence,
\[
|B|^{-1}\lesssim 2^{-3n\al}.
\]
 Taking sup of
\eqref{est3.2} over $B \in \cB_n$, we obtain
\[
a_n(t)+2b_n(t)
\leq Ca_n(0) +C2^{-2n\al} \int_0^t a_n\,ds
+  C_\e 2^{2n\al(\la-2)}\int_0^t a_n^3 \,ds + C \e b_n(t).
\]
Choosing $\e=1/C$, we can absorb the last term by the left side and get \eqref{grad.est.1}.
\end{proof}

\begin{remark}\label{la<2}
The condition $\la \le 2$ is essential in taking the sup of \eqref{est3.2}. If it fails, we cannot not have the a priori bound \eqref{grad.est.1}.
\end{remark}
\smallskip
Next we will use the foollowing version of Gr\"onwall lemma.
\begin{lem}\label{Gronwall}
  Suppose $f(t)\in L^{\infty}([0,T];[0,\infty))$ satisfies, for some $m\ge 1$,
  \EQ{
  f(t)\le a+\int_{0}^{t}(c_1f(s)+c_2f(s)^m)~ds~~\text{for}~ 0<t<T,
  }
  where $a,c_1,c_2>0$ then for $T_0=\min(T,T_1)$, with
  $$
  T_1=\frac{a}{2ac_1+(2a)^mc_2}
  $$
  we have $f(t)\le 2a$ for $t\in(0,T_0).$
\end{lem}
The proof is essentially that of \cite[Lemma 2.2]{BT8} although we do not take $c_1=c_2$. Note that $f(t)$ may be discontinuous.

Our goal is to get uniform a priori estimates for local energy solutions in classes $F^2_{\al,\la}$:
\begin{thm}[\textit{A priori} estimates]\label{uniform.bound}
Assume the condition \eqref{alla-cond} for $\al$ and $\la$.
  Assume $u_0\in F^2_{\al,\la}(\R^n)$ is divergence free and let $(u,p)$ be a local energy solution with initial data $u_0$ on $\R^3\times(0,T)$, with
  \EQ{\label{eq4.6}
  \esssup_{0<s<t}a_n(s)+b_n(t)< \infty
  }
  for all $t<T$, for one fixed $n\in \NN_0$. Then there exist $C_0,C_1$ independent of $n$ such that
 \EQ{\label{eq3.6}
  \esssup_{0<s<T_n^*}a_n(s)+b_n(T_n^*)\le C_0\|u_0\|_{F^2_{\al,\la}(R=2^n)}^2,
  }
where $T_n^* = \min(T, T_n)$ and
\EQ{\label{Tn.def}
T_n=C_1\bke{ 2^{-2n\al} + 2^{2n\al(\la-2)} \|u_0\|_{F^2_{\al,\la}(R=2^n)}^4}^{-1}.
}
\end{thm}
\begin{proof}
  We will apply Lemma \ref{Gronwall} to the apriori estimate we got in Lemma \ref{a.priori}:
  \EQN{
  a_n(t)\le C a_n(0)+\int_{0}^{t}(c_1(n)a_n(s)+c_2(n)a_n(s)^3)\,ds,
 }
where $c_1= C 2^{-2n\al}$ and $c_2=C 2^{2n\al(\la-2)}$.
  Therefore due to Lemma \ref{Gronwall} we get that
  \EQN{
   a_n(t)\le 2Ca_n(0),\qquad \forall 0<t<\min(T,T_n)
  }
 where $T_n=[2c_1+2(2Ca_n(0))^2c_2]^{-1}$,
 with $C$ independent of $n,u_0$.
The bound for $b_n(t)$ then follows from \eqref{grad.est.1} of Lemma \ref{a.priori}.
\end{proof}

\section{Local and global existence of weak solutions}
\label{sec5}
We now prove our main theorem,  Theorem \ref{mainthm}, on the local and global existence of weak solutions with initial data from $\mathring F^2_{\al,\la}$.

\begin{proof}[Proof of Theorem \ref{mainthm}]
If $\al  =0$, then we are back to local energy solutions in $E^2$ constructed by \cite{LR}, with details given by \cite{KiSe}.
In the following we assume $\al  >0$.

We first show local existence. Take $j\in \Bbb N$. Since $u_0\in \mathring F^2_{\al,\la,\si}$ there exists $u_0^j\in L^2_\si(\R^3)$ such that $\|u_0-u_0^j\|_{F^{2}_{\al,\la}(R=1)}\le\frac{1}{j}$. Let $(u^j,\bar{p}^j)$ be a global Leray solution for initial data $u_0^j$ in the energy class. By Theorem \ref{uniform.bound}, for $n \in \NN_0$, $u^j$ is bounded uniformly in $L^{\infty}(0,T_n;F^2_{\al,\la}(2^{n}))$ and satisfies
\EQ{\label{sequence.bound-loc}
\esssup_{0<s<T_n}\sup_{B\in \cB_n} \frac 1 {|B|^{\frac\la3}}\int_B |u^j(s)|^2 +\sup_{B\in \mathcal{B}_n}\frac{1}{|B|^{\frac\la3}}\int_0^{T_n}\!\!\!\int_{B}|\nabla u^j|^2\le C_0\|u_0\|_{F^2_{\al,\la}(2^{n})}^2
+ \e_{n,j},
  }
where $T_n$ given by \eqref{Tn.def}, and $\e_{n,j}=C_0\|u_0-u_0^j\|_{F^2_{\al,\la}(2^{n})}^2 \to0$ as $j\to\infty$ by Lemma \ref{vanishing.initial.data}.

Now let $T_n^* = \max_{j=1}^n T_j = T_{j(n)}$ for some $j(n) \le n$. Let $T_+=\sup_{n \in \NN} T_n$. We have $\lim_{n \to \infty} T_n^* =T_+$.
By the uniform bound \eqref{sequence.bound-loc}, the sequence $u^j$ is also bounded locally in energy space on $B \times (0,T_n^*)$ for every ball $B\subset \R^3$, since $B$ can be covered with finite number of balls $B_{r_x}(x)$ corresponding to the space $F^2_{\al,\la}(2^{j(n)})$. Denote $u^{0,k} = u^k$. By the uniform bound and induction, for any $n\in\Bbb N$
we can construct a subsequence $\{u^{n,k}\}_{k\in\Bbb N}$ from $\{u^{n-1,k}\}_{k\in\Bbb N}$ which converges to a vector field $u_n$ on $B_{2^{n}}\times(0,T)$ as $k\rightarrow\infty$ in the following sense:
\EQ{\gathered
u^{n,k}\overset{*}{\rightharpoonup} u_n~\textrm{in} ~L^{\infty}(0,T_n^*;L^2(B_{2^{n}})),
\\
u^{n,k}\rightharpoonup u_n~\textrm{in} ~L^2(0,T_n^*;H^1(B_{2^{n}})),
\\
u^{n,k}\rightarrow u_n~\textrm{in} ~L^{3}(0,T_n^*;L^3(B_{2^{n}})).
\endgathered}
Denote $\tilde{u}_n$ as a $0$ extension of $u_n$  to $\R^3\times (0,\infty)$. From our construction $\tilde{u}_n=\tilde{u}_{n-1}$ on $B_{2^{n-1}}\times(0,T)$. Let $u=\lim_{n\rightarrow\infty} \tilde{u}_n$. Then, $u=u_n$ on $B_{2^{n}}\times(0,T_n^*)$ for every $n\in\Bbb N$.
Let $u^{(k)}=u^{k,k}$ on $B_{2^{k}}\times(0,T_n^*)$ and $0$ elsewhere. Then for any $n\in\Bbb N$ we have the following convergence as $k\rightarrow\infty$:
\EQ{\label{ukconverge.loc}
\gathered
u^{(k)}\overset{*}{\rightharpoonup} u~\textrm{in} ~L^{\infty}(0,T_n^*;L^2(B_{2^{n}})),
\\
u^{(k)}\rightharpoonup u~\textrm{in} ~L^2(0,T_n^*;H^1(B_{2^{n}})),
\\
u^{(k)}\rightarrow u~\textrm{in} ~L^{3}(0,T_n^*;L^3(B_{2^{n}})).
\endgathered}
Notice that our sequence $\{u^{(k)}\}$ still satisfies \eqref{sequence.bound-loc} for all $n$ up to time $T_n$, therefore due to convergence \eqref{ukconverge.loc} the limit function $u$ also enjoys the same bound for every $n \in \NN_0$
 \EQ{\label{apriori}
\esssup_{0<s<T_n}\sup_{B\in \cB_n} \frac 1 {|B|^{\la/3}}\int_B |u(s)|^2 +\sup_{B\in \mathcal{B}_n}\frac{1}{|B|^{\la/3}}\int_0^{T_n}\!\!\!\int_{B}|\nabla u|^2\le C_0\|u_0\|_{F^2_{\al,\la}(2^{n})}^2.
}
\medskip

The pressure is dealt similarly to \cite[Section 3]{KwTs} and \cite[Section 4]{BKT}. From pressure estimate Lemma \ref{lemma.pressure}, as well as \eqref{ukconverge.loc} we know that associated pressure sequence $p^{(k)}$ converge in the following sense for any $n$
\EQ{\label{pkconverge.loc}
p^{(k)}\rightharpoonup p~\textrm{in} ~L^{3/2}((0,T_n^*)\times B_{2^{n}}),
}
and the limit satisfies the Navier-Stokes equation \eqref{eq.NSE}. Next we will establish pressure decomposition and its convergence. We follow the same argument as in \cite{BKT}: For any $B\subset\R^3$, $T<T_+$, and any $k\in\NN$ sufficiently large, there exists $p^{(k)}_B(t)\in L^{3/2}(0,T)$ such that $\forall t<T$ we have the following:
\EQ{
p^{(k)}(x,t)-p^{(k)}_B(t)=G_{ij}^B(u_i^{(k)}u_j^{(k)}),
}
where
\EQ{\gathered
G_{ij}^B(u_i^{(k)}u_j^{(k)})=-\frac13|u^{(k)}(x)|^2+p.v.\int_{y\in B^{**}}K_{ij}(x-y)(u^{(k)}_i(y,s)u^{(k)}_j(y,s))~dy
\\
+\int_{y\notin B^{**}}(K_{ij}(x-y)-K_{ij}(x_B-y))(u^{(k)}_i(y,s)u^{(k)}_j(y,s))~dy.
\endgathered}
Here we apply the convergence of $u^{(k)}$ in energy space, properties of singular integrals, and the a priori bounds \eqref{sequence.bound-loc} and \eqref{apriori}
 to prove that
\EQ{\label{G.convergence}
G_{ij}^B(u_i^{(k)}u_j^{(k)})\rightarrow G^B_{ij}(u_iu_j) ~\textrm{in}~ L^{\frac{3}{2}}(0,T;L^{\frac{3}{2}}(B)).
}
Indeed, choose $n$ so  {that $T\le T_n$, choose $N \ge n$
large enough so that $B^{**}\subset B(0,2^n)$,} and split $G_{ij}^B(u_i^{(k)}u_j^{(k)})=M_1^{(k)}+M_2^{(k)}+M_3^{(k)}$, where
\EQN{
M_1^{(k)}&=-\frac13|u^{(k)}(x)|^2+p.v.\int_{y\in B^{**}}K_{ij}(x-y)(u^{(k)}_iu^{(k)}_j(y,s))
\\
M_2^{(k)}&=\int_{y\in B(0,2^N)\setminus B^{**}}(K_{ij}(x-y)-K_{ij}(x_B-y))(u^{(k)}_iu^{(k)}_j(y,s))~dy
\\
M_3^{(k)}&=\sum_{m=N}^{\infty} \int_{y\in B(0,2^{m+1})\setminus B(0,2^{m})}(K_{ij}(x-y)-K_{ij}(x_B-y))(u^{(k)}_iu^{(k)}_j(y,s))~dy,
} 
where $N\in \NN$ and $N\ge n$.
We will use similar notations $M_1,M_2,M_3$ but for function $u$ instead of $u^{(k)}$. First we will prove smallness of $|M_3^{(k)}-M_3|$. After applying annulus covering Lemma \ref{annulas-integral} using balls from $\cB_n$, we get that
\EQN{\label{5.10}
&\sum_{m=N}^{\infty} \int_{y\in B(0,2^{m+1})\setminus B(0,2^{m})}(K_{ij}(x-y)-K_{ij}(x_B-y))\big |u^{(k)}_iu^{(k)}_j(y,s)-u_iu_j(y,s)\big|~dy
\\
&\lec \sum_{m=N}^{\infty}\int_{y\in B(0,2^{m+1})\setminus B(0,2^{m})}\frac{|B|^{1/3}}{2^{4m}}\big|u^{(k)}_iu^{(k)}_j(y,s)-u_iu_j(y,s)\big|~dy
\\
&\lec \sum_{m=N}^{\infty}\frac{|B|^{1/3}}{2^{4m}}2^{m(\al\la+3-3\al)}(\|u\|^2_{F^2_{\al,\la}(2^n)}+\|u^{(k)}\|^2_{F^2_{\al,\la}(2^n)})\lec |B|^{1/3}2^{N(\al\la-1-3\al)}\|u_0\|^2_{F^2_{\al,\la}(2^n)}.
} 
In the last inequality we have used $s\le T\le T_n$, the a priori bounds \eqref{sequence.bound-loc} and \eqref{apriori}, and taking $k$ sufficiently large (independent of $\e$) so that $\e_{n,k}$ in \eqref{sequence.bound-loc} is bounded using Lemma \ref{vanishing.initial.data},
\EQN{
\e_{n,k}&=C_0\|u_0-u_0^k\|^2_{F^2_{\al,\la}(2^n)} \le C_0 \|u_0-u_0^k\|^2_{F^2_{\al,\la}(1)} 2^{n(1-\al)\al(3-\la)} 
\\&\le C_0 k^{-2}  2^{n(1-\al)\al(3-\la)}  \le \|u_0\|^2_{F^2_{\al,\la}(2^n)}.
}
Therefore, for any $\varepsilon>0$ we can find large enough $n$ so that
\EQS{
\|M_3^{(k)}&-M_3\|_{L^{3/2}((0,T)\times {B})}\le C(B,T)2^{N(\al\la-1-3\al)}\|u_0\|_{F^2_{\al,\la}(2^n)}^2
\\
& \le C(B,T)2^{N(\al\la-1-3\al)}\|u_0\|_{F^2_{\al,\la}(1)}^2 C2^{n(1-\al)\al(3-\la)}  \le\varepsilon
}
using Lemma \ref{vanishing.initial.data} again, $N \ge n$, and $(\al\la-1-3\al)+(1-\al)\al(3-\la) <0$.
After choosing $N$ we apply \eqref{ukconverge.loc} for $M_1^{(k)},M_2^{(k)}$ to find $k$ large enough so that
\EQ{
\|M_1^{(k)}-M_1\|_{L^{3/2}((0,T)\times {B})}+\|M_2^{(k)}-M_2\|_{L^{3/2}((0,T)\times{B})}\le\varepsilon.
}
This proves convergence \eqref{G.convergence}. Lastly, in any fixed domain $B\times[0,T]$ pair $(u,G^B_{ij}(u_iu_j))$ solves \eqref{eq.NSE}. Therefore, $\nabla p=\nabla G_{ij}^B(u_iu_j)$ in $\mathcal{D}'(\R^3)$ at every time $t\le T$ and so there exists a function of time $p_B(t)$ such that
$$
p(x,t)= G_{ij}^B(u_iu_j)(x)+p_B(t), ~\text{for} ~(x,t)\in B\times [0,T].
$$
Also from this identity we get $p_B(t)\in L^{\frac{3}{2}}(0,T)$. This finishes proof of pressure decomposition and convergence.

\medskip

Next we need to establish local energy estimate using the above convergence  \eqref{ukconverge.loc} for $u^{(k)}\rightarrow u$ and  \eqref{pkconverge.loc} for $p^{(k)}\rightarrow p$ in $L^{3/2}_\loc$. Since $\{u^{(k)}\}$ is a subsequence of $\{u^j\}$ which are global Leray solutions we have local energy inequalities for $(u^{(k)},p^{(k)})$. For any $\phi \in C^{\infty}_c(\R^3\times (0,T_+)),~\phi\ge0,$
\EQS{
\int& |u^{(k)}(x,t)|^2\phi(x,t)~dx+ 2\int_0^t\int|\nabla u^{(k)}(x,s)|^2\phi(x,s) ~dx~ds
\\
&\le \int_0^t\int |u^{(k)}(x,s)|^2(\partial_t\phi(x,s)~dx~ds+\Delta \phi(x,s))~dx~ds
\\
&+\int_0^t\int(|u^{(k)}(x,s)|^2+2p^{(k)}(x,s))(u^{(k)}(x,s)\cdot \nabla\phi(x,s))~dx~ds.
}
We can find sufficiently large $n$ so that $\supp\{\phi\}\subset B_{2^n}\times[0,T_n^*]$ and applying the convergence of $u^{(k)}$ in energy space and $p^{(k)}$ in $L^{3/2}$ on $B_{{2^n}}$ we get the local energy inequality for $(u,p)$ which finishes the proof of local existence.
\medskip

For global existence, we want to ensure that $T_+=\limsup_{n \to \infty} T_n = \infty$.
By the definition \eqref{Tn.def} of $T_n$, it is equivalent to
\[
\liminf_{n \to \infty}  2^{n\al(\la-2)} \|u_0\|_{F^2_{\al,\la}(R=2^n)}^2 =0.
\]
It suffices to have \eqref{vanishing},
\[
\lim_{\rho \to \infty}  \rho^{\al(\la-2)} \|u_0\|_{F^2_{\al,\la}(\rho)}^2 =0.
\]
This condition is assumed when $\al(3-\la)<1$ by Theorem \ref{mainthm}. When $\al(3-\la)>1$, by Lemma \ref{vanishing.initial.data}
with $C_2 \lec \rho ^{\al(1-\al)(3-\la)}$, for $u_0\in F^2_{\al,\la}$ with $\|u_0\|^2_{F^{2}_{\al,\lambda}(R=1)}=A$,
\[
\rho^{\al(\la-2)} \|u_0\|^2_{F^{2}_{\al,\lambda}(\rho)} \lec \rho^{\al(\la-2)}  \rho ^{\al(1-\al)(3-\la)}
A= (\rho^\al)^{1-3\al+\al\la}A,
\]
which vanishes as $\rho \to \infty$. When $\al(3-\la)=1$, This condition is more delicate, and is proved by Lemma \ref{lem2.9} for $u_0 \in \mathring F^2_{\al,\la}$. This finishes the proof of global existence.
\end{proof}

\section{Eventual regularity}\label{sec6}

In this section we consider eventual regularity of our weak solutions, and give
 a proof of Theorem \ref{thm.ER}.
The proof requires a variant of the Caffarelli-Kohn-Nirenberg regularity  criteria \cite{CKN} due to Lin \cite{L98}; see also \cite{LS99}.

\begin{lem}[$\e$-regularity criteria]\label{thrm.epsilonreg}
For any $\si\in (0,1)$,
there exists a universal constant $\e_*=\e_*(\si)>0$ such that, if a pair $(u,p)$ is a suitable weak solutions of \eqref{eq.NSE} in
$Z_r=Z_r(x_0,t_0)=B_r(x_0)\times (t_0-r^2,t_0)$,
and
\[
\e^3=\frac 1 {r^2} \int_{Z_r} (|u|^3 +|p|^{3/2})\,dx\,dt <\e_*,
\]
then $u\in L^\I(Z_{\si r})$.
Moreover,
\[
\|  \nabla^k u\|_{L^\I(Z_{\si r})} \leq C_k {\e}\, r^{-k-1}
   , k\in {\mathbb N}_0
\]
for universal constants $C_k=C_k(\si)$.
\end{lem}

There are many variants of Lemma \ref{thrm.epsilonreg}, which involve different local space-time norms. The following lemma demonstrates how to interpolate such space-time norms.

\begin{lem} \label{lem.interp}Fix $x$ and $t$ and suppose $B(x,\sqrt t)\subset B(x,r_x) \in \mathcal B_n$. Let $\th=3(1/2-1/p)\in [0,1]$  {and $\th q \le 2$}. 
Then,
\EQN{
\int_0^t \bigg( \int_{B(x,\sqrt t)} |u|^p\,dx\bigg)^{\frac q p} \,ds&\lesssim 
 {r_x^{ \la q/2} \,t^{1-\th q/2}  \bke{a_n(t)^{(1-\th)q/2} \, b_n(t)^{\th q/2} + a_n(t)^{q/2}}.}
}
\end{lem} 
\begin{proof}
Let $B=B(x,\sqrt t)$.
By the Gagliardo-Nirenberg and H\"older inequalities we have 
\EQN{
\int_0^t \bigg( \int_{B(x,\sqrt t)} |u|^p\,dx\bigg)^{\frac q p} \,ds &\lesssim t^{1- \th q/2} \sup_{0<\tau <t}\| u(\tau)\|_{L^2(B)}^{(1-\th)q/2}  \bigg(  \int_0^t \| \nb u\|_{L^2(B)}^2 \bigg)^{\th q /2}
\\&+|B|^{q(1/p-1/2)}\int_0^t  \| u \|_{L^2(B)}^q \,ds.
}
Inserting the correct powers of $r_x$  and noting $B(x,\sqrt t)\subset B(x,r_x) \in \mathcal B_n$ completes the proof. 
\end{proof}

Our strategy will be to divide the estimate in Lemma \ref{lem.interp} by $t^{1+q(3/p-1)/2}$, which makes the left-hand side dimensionless. On the right-hand side, this leads to the following pre-factor on the leading order term when $|x|\geq 2^n$,
\[
{|x|^{\al\la q/2} }\,t^{-q(3/p-1)/2-\th q/2}.
\]
To close our argument, this will need to be bounded by $1$. Working out the cancellations results in the requirement that
\[
|x|^{2\al \la} \leq t.
\]
This condition is independent of $p$ and $q$. We will therefore choose $p=q=3$ in what follows.

The remainder of this section is entirely devoted to the proof of Theorem \ref{thm.ER}.

\begin{proof}[Proof of Theorem \ref{thm.ER}] Fix $x$ and $t$ and suppose $B(x,\sqrt t) \subset B(x,r_x)\in \mathcal B_n$. This holds, e.g., if $\sqrt t \leq 2^{n\al-3}$, which we assume to be true.
We have by Lemma \ref{lem.interp}  and Theorem \ref{uniform.bound} that 
\EQS{\label{eq6.2}
\frac 1 {  t}\int_0^t \int_{B(x,\sqrt t)} |u|^3\,dx\,ds &\lesssim r_x^{3\la/2  }t^{-3/4}  a_n(t)^{ 3/4} b_n(t)^{ 3/4} +  r_x^{3\la /2}  t^{ -3/4 } a_n(t)^{3/2}
\\&\lesssim %
r_x^{3 \la /2}  t^{-3/4}  a_n(0)^{3/2}.
}
In the preceding application of Theorem \ref{uniform.bound}, we need to have  $T_n \geq t$.
This is true if $T_n\geq 2^{2n\al-6}$, at least for large $n$.
This holds provided $u_0\in \mathring{\mathfrak F}^{2}_{\al,\la}$, $\la\leq 1$, as can be seen by the estimate%
\EQN{
64  C_1 &\geq 2^{2n\al }(2^{-2n\al} + 2^{2n\al (\la-2)} \|u_0\|_{F^2_{\al,\la}({R=2^n})}^4)
=1+ 2^{2n\al (\la-1)}  \|u_0\|_{F^2_{\al,\la}({R=2^n})}^4.
}
For the vanishing of \eqref{eq6.2}, we require
\[
r_x^{\frac {3 \la} 2 }t^{-3/4}  \leq 1, \mbox{ i.e.,  }r_x^{2\la}\leq t.
\]
Note that $a_n(0)^{3/2}\to 0$ because $u_0\in \mathring{ \mathfrak F}^{2}_{\al,\la}$.  

Before we apply $\epsilon$-regularity we must similarly bound the pressure.
The pressure is controlled using Lemma \ref{lemma.pressure} noting that this implies, for $\bar B = B(x,r_x)$, that
\EQN{
&\frac 1 { t }\int_0^{t}\int_{B(x,\sqrt t)} |p -p_B(s)|^{3/2}\,dy\,ds \\&\qquad \lesssim \sup_{B'\in \mathcal B_n; B'\cap \bar B^{**}\neq \emptyset} \frac {1}{t}\int_0^{t} \int_{B'} |u|^3\,dy\,ds
 +\frac 1 t  |\bar B|^{\frac \la 2-\frac 1 2}  a_n(0)^{3/2}.
}
Note that we  replaced $p$ in Lemma \ref{thrm.epsilonreg}   with $p-p_B(s)$ because $\nb p$ in \eqref{eq.NSE} is equal to $\nb (p-p_B(s))$.
Observe that if $B'\in \mathcal B_n$, $B'\cap \bar B^{**}\neq \emptyset$, and $y$ is the center of such a ball,  then $|y|\lesssim |x|$ with the suppressed constant independent of $n$.  
Applying the Gagliardo-Nirenberg inequality to $B'$,    we have for $t<T_n$ that
\EQN{
 \int_0^t \int_{B'} |u|^3\,dx\,ds &\lesssim |B'|^{\frac \la 2 }t^{1/4} a_n(t)^{3/4}b_n(t)^{3/4} + |B'|^{\frac {\la-1} 2} t a_n(t) ^{3/2}.
}
Hence,
\[
 \frac {1}{t}\int_0^{t} \int_{B'} |u|^3\,dy\,ds \lesssim r_x^{\frac {3\la} 2 }t^{-3/4} a_n(0)^{3/2}+ r_x^{3\frac {\la-1} 2}  a_n(0) ^{3/2}.
\]
Because $r_x^{2\la}\leq t$, $\la\leq 1$ and $u_0\in \mathring {\mathfrak F}^{2}_{\al,\la}$, the preceding terms can be made small by taking $n$ large.

Combining these estimates we see that
\[
	\frac 1 {t} \int_0^{t} \int_{B(x,\sqrt t)} |u|^3\,dy\,ds	
 + \frac 1 {t}\int_0^{t}\int_{B(x,\sqrt t)} |p -p_B(s)|^{3/2}\,dy\,ds < C a_n(0)^{3/2}
< \e_*,
\]
provided   $n$ is large enough and $r_x^{2\la}\leq t\leq 2^{2n\al -6}$. Let $n_*\in \N$ be the smallest such value of $n$ so that the second inequality holds.
Then, applying Theorem \ref{thrm.epsilonreg}  in $B(x,\sqrt t) \times [0, t]$ for $ |x|^{\al \la} \leq \sqrt t$ implies $u$ is regular in
$B(x,\si \sqrt t )\times [ (1-\si^2)t   , t]$, where we will specify $\si$ momentarily{, with
\[ \|u(x,t)\|_{L^\I (B(x,\si \sqrt t )\times [ (1-\si^2)t   , t])} \lesssim (\si \sqrt t)^{-1}.\]
} 
This holds in particular for $t = 2^{2n_*\al -6}$.
Repeating this for $n\geq n_*$ we find that $u$ is regular in%
\[
\bigcup_{n\geq n_* }  \bigcup_{|x|^{2\la\al}\leq  2^{2n\al -6 }} B(x,\si 2^{n \al  - 3} )\times [(1-\si^2) 2^{2n\al-6}, 2^{2n\al-6}].
\]
Choose $\si$ so that%
\[
(1-\si^2)2^{2n\al  -6}= 2^{2(n-1)\al-6},
\]
which ensures there are no time-gaps in the preceding double union. 
If $(x,t)$ satisfies $|x|^{2\al\la}\leq  t$, $t>M$ and $n$ is chosen so that
$(1-\si^2)2^{2n\al  -6} \leq t \leq 2^{2n\al-6}$, then $(x,t) \in \cup_{|x|^{2\la\al}\leq  2^{2n\al -6}}  B(x,\si 2^{n \al  - 3} )\times [(1-\si^2) 2^{2n\al-6}, 2^{2n\al-6}]$. Hence $u$ is regular within the set
\[
\{ (x,t): |x|^{2\al\la} \leq t; t\geq M\},
\]
where  $M= (1-\si^2)2^{2n_*\al-6}$ { and additionally satisfies $|u(x,t)|\lesssim \sqrt t^{-1}$}. This finishes the proof of Theorem  \ref{thm.ER}. 
\end{proof}

\begin{remark}
When $1<\lambda \leq 2$, observe that in the argument above we required the following:
\EQ{
\lim _{n\to \infty} 2^{n\al (\la-1)}a_n(0)=0;
\quad
\lim _{n\to \infty} r_x^\la t^{-1/2}a_n(0)=0;
\quad
\lim _{n\to \infty} r_x^{\la-1}a_n(0)=0.
}
If
\[
a_n(0) \le o(1) 2^{n\al(1-\la)},
\]
then we have smallness when
\[
r_x \le 2^{n \al}, \quad r_x^\la t^{-1/2} \lec 2^{n \al(\la-1)}.
\]
This leads to an algebraic region of the form,  {for example,}%
\[
\{ (x,t); |x|^{2\al} \lesssim t;\ 1\lesssim t   \}.
\]
However, this does not give a  new result compared to  Theorem  \ref{thm.ER} because  the condition \[
\lim _{n\to \infty} 2^{n\al (\la-1)}a_n(0)=0,
\]
implies $u_0 \in\mathring{ \mathfrak F}^2_{\al,1}$, as we presently explain.
 Note that  $a_n$ is defined with respect to $F^{2}_{\al,\la}$. Suppose $f$ satisfies the preceding condition. If   $|x|\leq 2^n$, then  
\[
\frac 1 {2^{n\al}}\int_{B(x,2^{\al n}/8)} |f|^2 \,dx= \frac {2^{n\al (\la-1)}} {2^{n\al \la }}  \int_{B(x,2^{\al n}/8)} |f|^2 \,dx \leq 2^{n\al (\la-1)}a_n(0).
\]
On the other hand, if 
$|x|>2^n$, then let $k\ge n$ satisfy $2^k\leq |x| <2^{k+1}$. 
Then,
\[
\frac 1 {|x|^\al} \int_{B(x,|x|^\al/8)}|f|^2\,dx \lesssim \frac 1 {2^{\al k}} \int_{B(x,|x|^\al/8)}|f|^2\,dx.
\]
We know  $\sup_{k\geq n}2^{k\al (\la-1)}a_k(0)<\I$. So, 
\[
\frac 1 {2^{\al k}} \int_{B(x,|x|^\al/8)}|f|^2\,dx \lesssim  \frac {2^{k\al (\la-1)}} {2^{k\al \la }}\int_{B(x,2^{\al (k+1)}/8)}|f|^2\,dx \leq \sup_{k\geq n}2^{k\al (\la-1)}a_k(0)  <\I.
\]
Hence, $f\in F^{2}_{\al,1}(2^n)$ with 
\EQ{
\| f \|^2_{F^{2}_{\al,1}(2^n)} \leq \sup_{k\geq n}2^{k\al (\la-1)}a_k(0).
}
Observing that the right-hand side above vanishes as $n\to \I$, we conclude  {$f\in\mathring{ \mathfrak F}^2_{\al,1}$.}
\end{remark}

\section*{Acknowledgments}
The research of ZB was supported in part by NSF grant DMS-2307097.
The research of both MC and TT was supported in part by
NSERC grants RGPIN-2018-04137 and RGPIN-2023-04534.
MC was also supported in part by a UBC 4-year fellowship.

\addcontentsline{toc}{section}{\protect\numberline{}{References}}

\medskip
Zachary Bradshaw, Department of Mathematics, University of Arkansas, Fayetteville, AR 72701, USA;
e-mail: zb002@uark.edu
\medskip

Misha Chernobai, Department of Mathematics, University of British Columbia, Vancouver, BC V6T 1Z2, Canada;
 e-mail: mchernobay@gmail.com

\medskip

Tai-Peng Tsai, Department of Mathematics, University of British Columbia,
Vancouver, BC V6T 1Z2, Canada; e-mail: ttsai@math.ubc.ca


\begin{thebibliography}{[28]}
%
%
%
%
\bibitem{Basson} A. Basson, \textit{Solutions spatialement homog\'enes adapt\'ees au sens de Caffarelli, Kohn et Nirenberg des \'equations de Navier-Stokes,} Th\'ese, Universit\'e d’Evry, 2006.

\bibitem{BK} Z.~Bradshaw and I.~Kukavica,
\emph{Existence of suitable weak solutions to the Navier-Stokes equations for intermittent data},
J.~Math.\ Fluid Mech.~22 (2020), no. 1, Art. 3, 20 pp.

\bibitem{BKT} Z.~Bradshaw, I.~Kukavica, and T.-P.~Tsai,
\emph{Existence of global weak solutions to the Navier-Stokes equations in weighted spaces},
Indiana University Mathematics Journal
%
71 (2022), no. 1, 191-212
%

%
%
%
%
%
%
%
%
%
%
%


\bibitem{BKO} Z. Bradshaw, I. Kukavica and W. Ozanski, \textit{Global weak solutions of the Navier-Stokes equations for intermittent initial data in half-space}. Arch. Ration. Mech. Anal. 245 (2022), no.1, 321–371.

\bibitem{BLT}
  Z.~Bradshaw, C.-C.~Lai and T.-P.~Tsai,
\emph{Mild solutions and spacetime integral bounds for Stokes and Navier-Stokes flows in Wiener amalgam spaces}, Mathematische Annalen (2023).

\bibitem{BT8} Z.~Bradshaw and T.-P.~Tsai,
   \emph{Global existence, regularity, and uniqueness of infinite energy solutions to the Navier-Stokes equations},
Commun. in Partial Differential Equations 45 (2020), no. 9, 1168-1201.

\bibitem{BT4}
  Z.~Bradshaw and T.-P.~Tsai,
\emph{Local energy solutions to the Navier-Stokes equations in Wiener amalgam spaces}, SIAM J. Math. Anal.  53 (2021) no. 2, 1993-2026.
%

\bibitem{Che}
M.~Chernobai, \textit{Existence of global weak solutions to NSE in weighted spaces}, preprint: 
\href{https://arxiv.org/abs/2310.17702}{arXiv:2310.17702}. 

%
%
%
%

%
\bibitem{CKN} %
 L.~Caffarelli, R.~Kohn, and L.~Nirenberg,
  \emph{Partial regularity of suitable weak solutions of the Navier-Stokes equations},
  Comm.\ Pure Appl.\ Math.~\textbf{35} (1982), no.~6, 771--831.
%

\bibitem{CW} D.~Chae and J.~Wolf,
\textit{Existence of local suitable weak solutions to the Navier-Stokes equations for initial data in  $L^2_{\loc}(\R^3)$ 
Nonlinear Anal. Real World Appl. 73 (2023), Paper No. 103892, 28 pp.}
%
%
%
%
%
%
%
%
%
%

%
%

%

\bibitem{FDLR2} P.G.~Fern\'andez-Dalgo and P.G.~Lemari\'e-Rieusset,
\textit{Characterisation of the pressure term in the incompressible Navier-Stokes equations on the whole space},
Discrete Contin. Dyn. Syst. Ser. S 14 (2021), no. 8, 2917–2931.


\bibitem{FDLR}
   P.G.~Fern\'andez-Dalgo and P.G.~Lemari\'e-Rieusset,
   \emph{Weak solutions for Navier-Stokes equations with initial data in weighted $L^2$ spaces},
 Arch. Ration. Mech. Anal. 237 (2020), no. 1, 347–382.
 %
%
%
%
%
  %

%
%
%
%
%
%
%
%
\bibitem{JiaSverak-minimal} %
H.~Jia and V.~\v{S}ver\'{a}k, \emph{Minimal {$L^3$}-initial data for potential {N}avier-{S}tokes singularities},
  SIAM J.~Math.\ Anal.~\textbf{45} (2013), no.~3, 1448--1459.
%
%
%
%
%

\bibitem{KiSe} %
  N.~Kikuchi and G.~Seregin,
  \emph{Weak solutions to the {C}auchy problem for the {N}avier-{S}tokes equations satisfying the local energy inequality},
  Nonlinear equations and spectral theory, Amer.\ Math.\ Soc.\ Transl.\ Ser.~2, vol.~220, Amer.\ Math.\ Soc., Providence, RI, 2007, pp.~141--164.
  %
%
%
%
%
%

%
%
%
%
%
%
%
%
%
%

\bibitem{KwTs}
  H.~Kwon and T.-P.~Tsai,
  \emph{Global Navier-Stokes flows for non-decaying initial data with slowly decaying oscillation},  Comm. Math. Phys. 375, 1665-1715 (2020).
%

\bibitem{LS99}
 O.A.~Lady{\v{z}}enskaja and G.~Seregin,
  \emph{On partial regularity of suitable weak solutions to the three-dimensional {N}avier-{S}tokes equations},
  J.~Math. Fluid Mech.~\textbf{1} (1999), no.~4, 356--387.
%
%
%
%
%
    %

%

\bibitem{LR}
P.~G.~Lemari{\'e}-Rieusset,
  \emph{Recent developments in the Navier-Stokes problem},
  Chapman \& Hall/CRC Research Notes in Mathematics, vol.~431, Chapman \& Hall/CRC, Boca Raton, FL, 2002.
  %


%

\bibitem{LR-Morrey} %
P.~G.~Lemari\'{e}-Rieusset,
  \emph{The {N}avier-{S}tokes equations in the critical {M}orrey-{C}ampanato space},
  Rev.\ Mat.\ Iberoam.~\textbf{23} (2007), no.~3, 897--930.
  %


\bibitem{LR2} %
P.~G.~Lemari\'e-Rieusset,
  \emph{The {N}avier-{S}tokes problem in the
  21st century}, CRC Press, Boca Raton, FL, 2016.
  %
%
\bibitem{leray}
  J.~Leray, \emph{Sur le mouvement d'un liquide visqueux emplissant l'espace},
  Acta Math.~\textbf{63} (1934), no.~1, 193--248.
%


\bibitem{L98}
F.~Lin, \emph{A new proof of the {C}affarelli-{K}ohn-{N}irenberg theorem},
  Comm. Pure Appl. Math.~\textbf{51} (1998), no.~3, 241-257.
%
%
%
%
%
%
%
\bibitem{MaMiPr} %
Y.~Maekawa, H.~Miura, and C.~Prange,
  \emph{Local energy weak solutions for the {N}avier-{S}tokes equations in the half-space},
  Comm.\ Math.\ Phys.~\textbf{367} (2019), no.~2, 517--580.
%
%

%

%
%
%
%

%
%
%
%
%
%
%
%
%
%
%
%


%
%
%
%
%
%
%


%
%
%
%
%




%
%
%
%
%
%
  %


\bibitem{Tsutsui} %
  Y.~Tsutsui,
  \emph{The {N}avier-{S}tokes equations and weak {H}erz spaces},
  Adv.\ Differential Equations~\textbf{16} (2011), no.~11-12, 1049--1085.
  %
\end{thebibliography}
\end{document}